%% file: main.tex
\documentclass[a4paper, 11pt]{article}
\input{macros.tex}

\title{Classification of Local Problems on Paths from the Perspective of Descriptive Combinatorics }

\begin{document}

\author{Jan Grebík \\ \small{University of Warwick} \\ \small{jan.grebik@warwick.ac.uk} \and Václav Rozhoň\\ \small{ETH Zurich} \\ \small{rozhonv@ethz.ch}
}

\maketitle

\begin{abstract}
    We classify which local problems with inputs on oriented paths have so-called Borel solution and show that this class of problems remains the same if we instead require a measurable solution, a factor of iid solution, or a solution with the property of Baire. 
    
    Together with the work from the field of distributed computing [Balliu et al. PODC 2019], the work from the field of descriptive combinatorics [Gao et al. arXiv:1803.03872, Bernshteyn arXiv:2004.04905] and the work from the field of random processes [Holroyd et al. Annals of Prob. 2017, Grebík, Rozhoň 	arXiv:2103.08394], this finishes the classification of local problems with inputs on oriented paths using complexity classes from these three fields. 
    
    A simple picture emerges: there are four classes of local problems and most classes have natural definitions in all three fields. 
    Moreover, we now know that randomness does \emph{not} help with solving local problems on oriented paths. 
    
    %Our approach combines techniques from all three fields. \todo{?H: to uz je receny v prvnim odstavci bych rekl V: ok}
\end{abstract}

\section{Introduction}\label{sec:intro}

\emph{Locally checkable problems (LCLs)} is a class of graph problems where the correctness of a solution can be checked locally. This class includes problems like vertex or edge coloring, perfect matching or a maximal independent set, as well as problems with inputs such as list colorings etc.
In this paper, we study LCLs on oriented paths from three different perspectives: the perspective of \emph{descriptive combinatorics}, \emph{distributed algorithms},  and \emph{random processes}.
Each perspective offers different procedures and computational power to solve LCLs, as well as different scales to measure complexity of LCLs.
In particular, the notion of easy/local vs hard/global problems varies in different settings.

In this work we completely describe the connections between these complexity classes and show that the problem of deciding to what class a given LCL belongs is decidable.
This extends the known classification of LCLs without inputs on oriented paths (see \cref{sec:no_inputs}) and is in contrast with the fact that on $\mathbb{Z}^d$, where $d>1$, this problem is not decidable.

Moreover, we hope that our work helps to clarify the relationships between the three different perspectives and, since the situation is both non-trivial and well understood, may serve as a basis of a common theory.
We note that for different classes of graphs, i.e., grids, regular trees, etc, describing possible complexity classes of LCLs and the relationship between them is not known, is already much harder and contains many exciting questions, see~\cite{grebik_rozhon2021toasts_and_tails,brandt_chang_grebik_grunau_rozhon_vidnyaszky2021LCLs_on_trees_descriptive}.

We will now briefly explain all three fields and then the particular setup of LCLs on paths that we are interested in. 
Then we state \cref{thm:main}, our main theorem, and \cref{cor:main}, the corrolary of \cref{thm:main} and previous work which is a classification of complexity classes of LCL problems on oriented paths that relates the three different perspectives to each other.
To make the presentation more precise, we introduce the formal definition of LCLs.

\begin{definition}[LCLs on oriented paths]
\label{def:lcls_on_paths}
A locally checkable problem (LCL) on an oriented path is a quadruple $\Pi=(\Sigma_{in}, \Sigma_{out}, r, \fP)$, where $\Sigma_{in}$ and $\Sigma_{out}$ are finite sets, $r$ is a positive integer, and $\fP$ is a function from the space of finite rooted $\Sigma_{in}$-$\Sigma_{out}$-labeled paths of diameter (at most) $r$ to the set $\{\texttt{true}, \texttt{false}\}$.

A \emph{correct solution} of an LCL problem $\Pi$ in a $\Sigma_{in}$-labeled oriented path $G$ (finite cycle or infinite path) is a map $f:V(G)\to \Sigma_{out}$ such that $\fP$ returns $\texttt{true}$ for $\Sigma_{in}$-$\Sigma_{out}$-labeled rooted $r$-neighborhood of every node $v\in V(G)$. 
Every such $f$ is called a \emph{$\Pi$-coloring}.
\end{definition}

Example of an LCL problem is $k$-coloring: there, $\Sigma_{in} = \emptyset$, $\Sigma_{out} = \{1, 2, \dots, k\}$, $r=1$, and $\fP$ returns true for all pairs $(\sigma_1, \sigma_2) \in \Sigma_{out}^2$ such that $\sigma_1 \not= \sigma_2$. 
Other examples of problems without inputs are: maximal independent set, edge coloring, or perfect matching. Problems with inputs include for example list coloring (we formally require the set of all colors in the lists to be finite). 

%In this work, we want to understand exactly these problems. 
The fact that we work with the input labelling is important. Without it, the whole classification problem is substantially simpler (\cref{sec:no_inputs}). The reason we care about inputs is twofold. First, problems with inputs contain more general setups such as the setup of the famous circle squaring problem that we discuss later. Second, understanding problems on lines with inputs serves as an intermediate step towards understanding problems on trees, at least in the distributed computing area \cite{balliu2019LCLs_on_paths_decidable,chang2017time_hierarchy}. 

\paragraph{Distributed Algorithms}

The study of the $\local$ model of distributed algorithm \cite{linial1987LOCAL} is motivated by understanding distributed algorithms in huge networks. 
As an example, consider all wifi routers on the planet. Think of two routers being connected if they are close enough to exchange messages. In this case, they should better communicate with user devices on different channels, so as to avoid interference. 
In the language of graph theory, we want to color the network of routers, in a distributed fashion. 
Even if the maximum degree of the network is $\Delta$ and we want to color it with $\Delta+1$ colors, the problem remains nontrivial, because the decision of each vertex (wifi router) should be done after few communication rounds with its neighbors. 

The $\local$ model formalizes this setup: there is a (huge) graph such that each of its nodes at the beginning knows only its size $n$, and perhaps some other parameter like the maximum degree $\Delta$. 
In case of randomized algorithms, each node has an access to a random string, while in case of deterministic algorithms, each node starts with a unique identifier from a range polynomial in the size of the graph $n$. 
In one round, each node can exchange any message with its neighbors and can perform an arbitrary computation. We want to find a solution to a problem in as few communication rounds as possible. 

Importantly, there is an equivalent view of $t$-round $\local$ algorithms: such an algorithm is simply a function that maps all possible $t$-hop neighbourhoods to the final output. An algorithm is correct if and only if applying this function to the $t$-hop neighborhood of each node solves the problem. 

The simplest possible setting to consider in the $\local$ model is if one restricts their attention to the simplest graph: a sufficiently long consistently oriented cycle. This graph is the simplest model, since all local neighborhoods of it look like an oriented path. 
This case is very well understood: It is known that any LCL problem can have only one of three complexities on oriented paths and the randomized complexity is always the same as the deterministic one. 
First, there are problems solvable in $O(1)$ local computation rounds -- think of the problems like ``color all nodes with red color'' or ``how many different input colors are there in my 5-hop neighborhood?''. 
%Note that the second problem operates with some additional input on each node but this is allowed in the definition of an LCL. 
Second, there is a class of basic symmetry breaking problems solvable in $\Theta(\log^* n)$ rounds -- this includes problems like $3$-coloring, list-coloring with lists of size at least $3$, or maximal independent set. 
Finally, there is a class of ``global'' problems that cannot be solved in $o(n)$ rounds -- these include e.g. $2$-coloring or perfect matching. 

\paragraph{Descriptive Combinatorics}

%It was very surprising when 
In 1990s Laczkovich resolved the famous Circle Squaring Problem of Tarski \cite{laczkovich}: A square of unit area can be decomposed into finitely many pieces that can be translated to form a disk of unit area.
This result was improved in recent years to make the pieces measurable in various sense \cite{OlegCircle,Circle,JordanCircle}.
This theorem and its subsequent strengthenings are the highlights of a field nowadays called descriptive combinatorics \cite{KST,kechris_marks2016descriptive_comb_survey,pikhurko2021descriptive_comb_survey} that has close connections to distributed computing as was shown in an insightful paper by Bernshteyn \cite{Bernshteyn2021LLL}. 

The simplest non-trivial setup of descriptive combinatorics is the following \cite{pikhurko2021descriptive_comb_survey}. Consider the unit cycle $\mathbb{S}^1$ in $\mathbb{R}^2$, i.e., the set of pairs $(x,y)$ such that $x^2+y^2=1$. Imagine rotating this cycle by a rotation $\alpha > 0$ that is irrational with respect to the full rotation, i.e., $\alpha / (2\pi) \not\in \mathbb{Q}$.
%is imagine the  which is an irrational multiple of the full rotation, that is, $\alpha / (2\pi) \not\in \mathbb{Q}$. 
This rotation naturally induces a directed graph $G_\alpha = (\mathbb{S}^1,E)$ on the cycle $\mathbb{S}^1$ where a directed edge in $E$ points from $(x,y)$ to $(x',y')$ if one gets $(x',y')$ by rotating $(x,y)$ by $\alpha$. 
This graph has uncountably many connected components, each of which is a directed path going to infinity in both directions. 

We could now ask questions like: What is the chromatic number of $G$?
If there were no other restrictions, the answer is $2$, i.e., color each path separately.
However, for this type of argument we implicitly use the axiom of choice.
The main goal of descriptive combinatorics is to understand what happens if some additional definability/regularity requirements are posed on the colorings.
That is, what if we require each color class to be an interval, open, Borel, or Lebesgue measurable set?
With any of those requirements, the chromatic number of $G$ now increases to $3$.

The reason is as follows. 
On one hand, one can color $G$ with three colors, if one starts by coloring a half-open interval $I$ of length $\alpha$ on $\mathbb{S}^1$ with color $1$. It can be seen that the rest of the cycle can be colored with colors $2$ and $3$ by shifting the interval by multiples of $\alpha$ (see \cref{fig:cycle}). 
On the other hand, a two coloring is impossible.
This is easily seen when we require our color classes to be unions of intervals.
Imagine such a two-coloring of $\mathbb{S}^1$ and consider an arbitrary node $x \in \mathbb{S}^1$ colored with, say, red color. Then, all the nodes that we get by rotating $x$ by a multiple of $2\alpha$ are red, too. But since $\alpha$, and therefore $2\alpha$, is irrational rotation we see that the set of red points is dense in $\mathbb{S}^1$ (see \cref{fig:cycle}).
Hence, it intersects any other interval, in particular the color class blue.
Similar argument rules out any potential Lebesgue measurable $2$-coloring, where the argument that red points are dense is replaced by the fact that irrational rotations are ergodic.

\begin{figure}
    \centering
    \includegraphics[width = .9\textwidth]{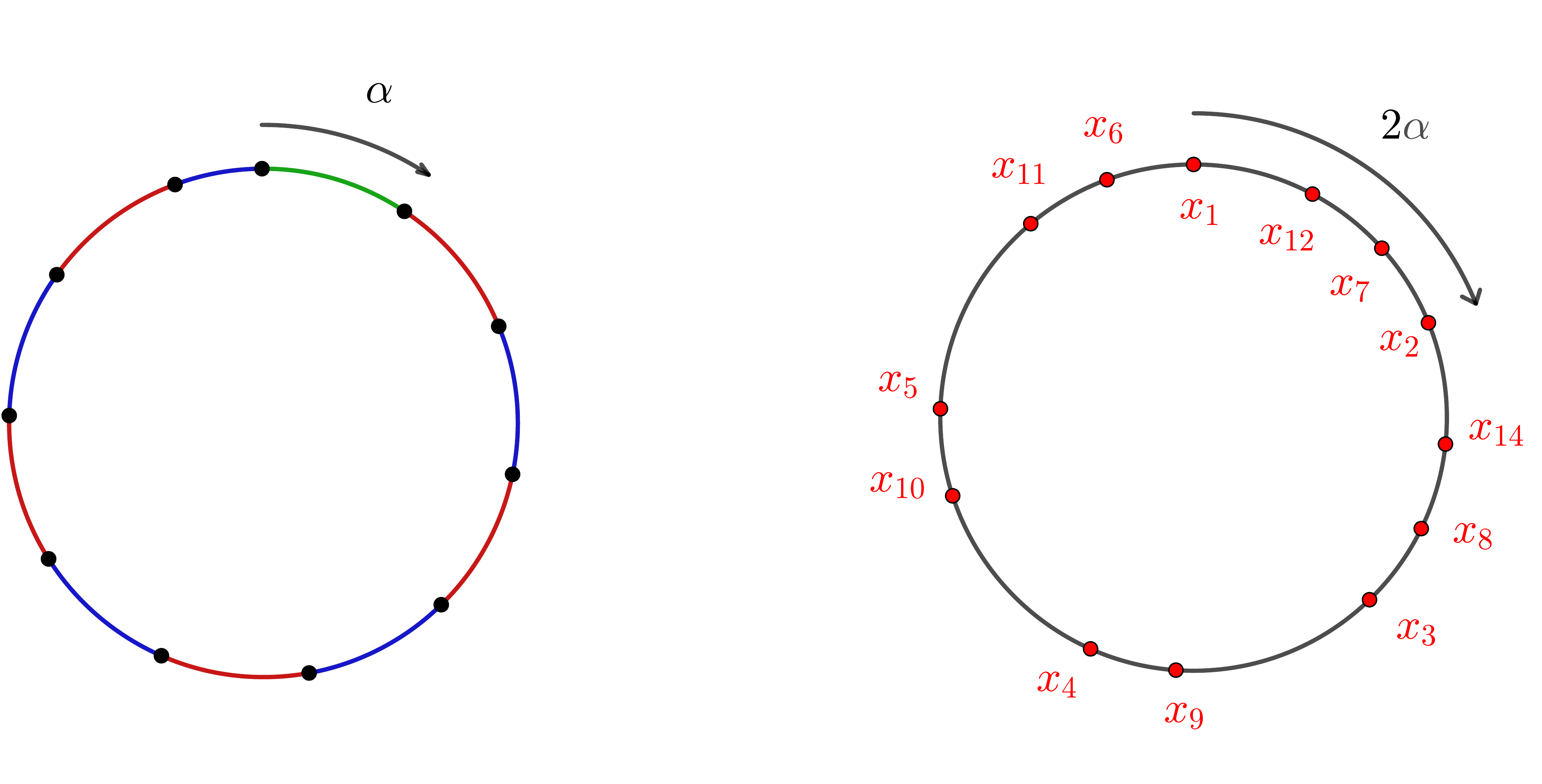}
    \caption{Left: 3-coloring by intervals is constructed by
    first coloring a half-open interval of length $\alpha$ with green and then rotating this interval and using alternatively a red and blue color. \\
    Right: if we start rotating a point $x_1$ by multiples of $2\alpha$, we get a dense set that intersects any (supposedly blue) interval. }
    \label{fig:cycle}
\end{figure}

For a general LCL without input labelings the situation is completely understood, see \cref{sec:no_inputs}.
In this case, problems solvable on $G$ (with the additional definability requirements) are exactly those that have local complexity $O(\log^* n)$.
Moreover, as in the case of $3$-coloring, if a problem is solvable, then one can find a solution using only unions of intervals.

Our main theorem generalizes to the setup with input labels.
In the example with $\mathbb{S}^1$ this means that an adversary first partitions $\mathbb{S}^1$ into sets that are indexed by input labels, i.e., $\mathbb{S}^1=\bigsqcup_{\sigma\in \Sigma_{in}}A_\sigma$.
Again, we require each input color class to satisfy some definability properties, e.g., union of intervals, open, Borel, or Lebesgue measurable set.
We may view any such partition as labeling of nodes simply by thinking that a node $x\in \mathbb{S}^1$ is labeled by $\sigma$ if $x\in A_\sigma$.
This induces a labeling of each oriented doubly infinite path of the graph $G$.
Now we ask for a definable coloring with colors from $\Sigma_{out}$ that would solve $\Pi$. 
As we demonstrate more LCL problems can be solved in this setup than with local algorithms.
However, the picture is still very clean and in our view nicely explains the power of descriptive combinatorics when compared with distributed algorithms.

The example with the circle $\mathbb{S}^1$ demonstrates in an instructive way what one should expect from definable constructions.
However, after one becomes familiar with this example and wants to understand Borel contructions in general, it is better to use a different view.
It is known that local algorithm with round complexity $O(\log^* n)$ can, in fact, use only $O(1)$ rounds if we first solve an MIS on a sufficiently large power of the input graph \cite{chang2017time_hierarchy} (cf. \cref{sec:no_inputs}).
Finding MIS for any power of the input graph is possible in the Borel context \cite{KST}.
This allows to think of a ``Borel algorithm'' as follows: the algorithm can use $O(1)$ local constructions and has an oracle access to MIS on graphs it can construct, but the algorithm can ``run'' for any countable number of steps. 
This is analogous to how Borel $\sigma$-algebra is defined: one starts with a set system of basic open sets and then creates more complicated sets by doing basic set operations for any countable number of steps. 
This already hints that such a model is much stronger then the $\local$ model: in the Borel setting each node has access to some global information on its connected component, e.g., how many input labels appear etc., and then use this knowledge to solve a given problem.

\paragraph{(Finitary) Factors of iid Processes}

To show that the proper vertex $2$-coloring on the graph $G_\alpha$ induced by an irrational rotation of the circle $\mathbb{S}_1$ does not exist with Lebesgue measurable sets, one uses the means of measure theory, in particular, of ergodic theory.
In its most abstract form, the goal of ergodic theory is to classify measure preserving transformations of probability measure spaces, e.g., by introducing invariants such as entropy or mixing properties.
A prominent collection of examples, that also plays important role in this paper and in the area of random processes, are the Bernoulli shifts, also called \emph{iid processes}, and their factors, \emph{fiid}.
In the area of random processes two natural questions are studied: (a) what are the possible factors of a iid processes \cite{HolroydSchrammWilson2017FinitaryColoring,grebik_rozhon2021toasts_and_tails} and (b) can a given process be described as a factor of a given iid process \cite{Spinka}.

An iid process is, for example, given by independent uniform labelings of vertices of $\mathbb{Z}$ by $2^{\mathbb{N}}$, i.e., infinite strings of $0$s and $1$s, together with the natural shift action given by the group structure of $\mathbb{Z}$.
Getting back to the proper vertex $2$-coloring problem, imagine that you want to find it as a factor of this iid process in the following way.
Start from the iid process and let each vertex $v\in \mathbb{Z}$ explore its neighborhood.
Each node is required to finish after finitely many steps and output one of the two allowed colors only depending on the $2^\mathbb{N}$-labels in this neighborhood in such a way that the global coloring is a proper vertex $2$-coloring almost surely.
Such a processes is called \emph{finitary factor of iid processes (ffiid)}.
Processes, when we do not require the exploring to finish after finitely many steps but allow each vertex to see the whole line, are exactly factors of iid processes (fiid).
It is a basic result in the area of random processes that every fiid needs to satisfy some correlation decay between decisions of far away points.
Note that the proper vertex $2$-coloring does not show such a behavior and therefore it cannot be described as fiid, nor ffiid.

The connections between LCL problems that admit ffiid solution and $\local$ algorithms was studied in \cite{HolroydSchrammWilson2017FinitaryColoring,grebik_rozhon2021toasts_and_tails} for $\mathbb{Z}^d$, $d>0$.
The class of such LCL problems is called $\tail$ while the class of problems that admit fiid solution is called $\FIID$.
The connection between these classes, when we restrict our attention to LCL problems without inputs, is not clear unless $d=1$.
In that case they are equal.
In the situation with inputs we have $\tail\not=\fiid$, see \cref{fig:big_picture_paths}.
It is an interesting open problem whether $\fiid=\borel$ on $\mathbb{Z}^d$ when $d>1$. See \cref{sec:Open_Problems_lines}.

\subsection{Our Contribution}

We extend the classification of LCLs on oriented paths from the perspective of distributed algorithms \cite{balliu2019LCLs_on_paths_decidable} to the perspective of descriptive combinatorics and random processes.
A simple picture emerges.
We find that with inputs the latter setups offer more complexity classes than distributed algorithms.
%those classes already contain more problems than classes from distributed and ffiid.

We now state our main theorem and its corollary which is a classification of LCL problems from three different perspectives. 
%Importantly, the classes $\local(f(n))$ contain LCLs that can be solved with a distributed algorithm with $O(f(n))$ local complexity. 
One should think of the class $\borel$ as in the example where the adversary partitions the circle that is described above and one is allowed to use a ``Borel algorithm'' to find a solution.
The classes $\measure$, $\baire$ are traditionally studied in descriptive combinatorics and offer more computational power. Intuitively, the additional power of $\measure$ when compared with $\borel$ is the same as the additional power of a randomized algorithm when compared with a deterministic one. The complexity classes below are formally defined in \cref{subsec:def_local}.

%In the following theorem, the normal form of an LCL problem is its form where the solution can be checked with locality $r = 1$ (\cref{def:normal_form_line}). 
%{See \cite{balliu2019LCLs_on_paths_decidable,brandt_grids,GaoJackson,Bernshteyn2021local=cont,HolroydSchrammWilson2017FinitaryColoring,grebik_rozhon2021toasts_and_tails}}
\begin{theorem}\label{thm:main}
For the classes of LCL problems on infinite oriented lines we have that
\[
\borel = \measure = \fiid = \baire. 
\]
Moreover, deciding whether $\Pi \in \borel$ is a $\pspace$-complete problem.  
\end{theorem}

%We comment that \cref{thm:main} as stated is a compilation of our result and work of Balliu et al \cite{balliu2019LCLs_on_paths_decidable}.
%Our contribution is the characterization of the class (C) and providing/summing up the connections through out different fields, these were often not explicitly stated anywhere or were considered a folklore. \todo{tady musime byt opatrni, citovat antona a tak}

As a corollary, this finishes the classification of complexity classes coming from the three analysed areas. 
\cref{cor:main} follows from the work in distributed algorithms 
\cite{chang2016exp_separation,chang2017time_hierarchy,balliu2019LCLs_on_paths_decidable,brandt_grids}, descriptive combinatorics \cite{GJKS,Bernshteyn2021LLL,Bernshteyn2021local=cont}, and finitary factors of iid processes \cite{HolroydSchrammWilson2017FinitaryColoring,grebik_rozhon2021toasts_and_tails,brandt_grebik_grunau_rozhon2021classification_of_lcls_trees_and_grids}.
This classification is complete in the sense that we are not aware of other used classes of problems. It describes a unified picture of locality for one particular class of graphs and it helps to clarify the computational powers and limits of different approaches. 
Most importantly, on high-level the distinction between classes (B) and (C) in \cref{cor:main} is that in (C) we have access to some global information about the input while allowed steps in (B) are purely local (see \cref{subsec:intuition_about_main}).  

\begin{figure}
    \centering
    \includegraphics[width=.97\textwidth]{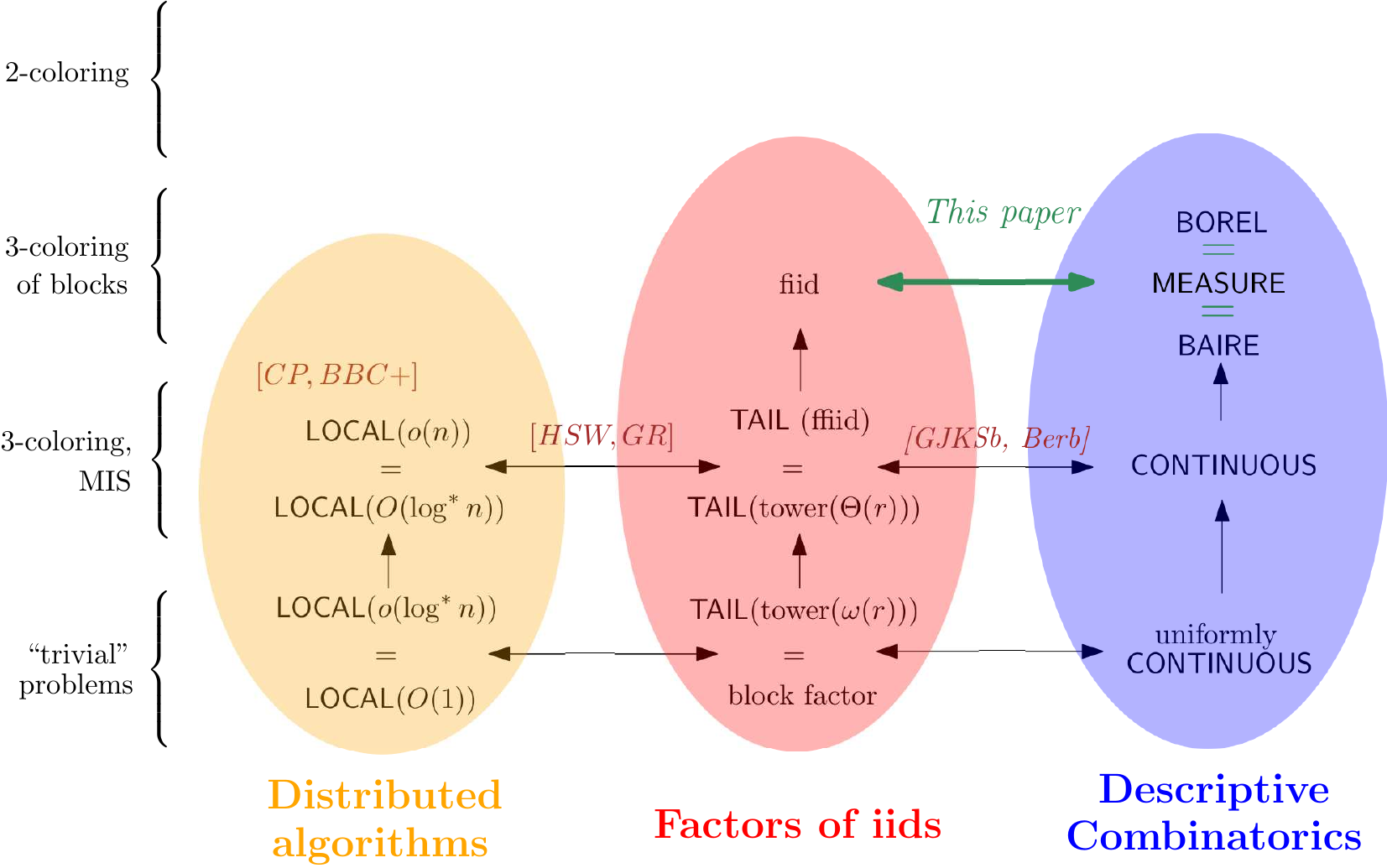}
        \caption{The figure shows a classification of LCL problems from \cref{cor:main} from all three considered perspectives. A clear picture emerges. There are four classes of problems. First, there is class that contains trivial problems such as ``how many different input colors are there in my 5-hop neighborhood''. Then, there is a class of basic symmetry breaking problems such as $3$-coloring or MIS. Then, there is a class of problems that we can solve with basic symmetry breaking tools, but we cannot do it locally. An example problem is the \blockcol problem. Finally, there is a class of ``global'' problems that contain e.g. $2$-coloring.  }
    \label{fig:big_picture_paths}
\end{figure}
\todo{gjks je jinak}

\begin{corollary}[Classification of LCL Problems on Oriented Paths, see \cref{fig:big_picture_paths}]\label{cor:main}
For LCL problems on infinite oriented lines, or large enough oriented cycles in case of classes $\local(O(f(n))$, we have the following complexity classes
\begin{itemize}
	\item [(A)] $\local(O(1))$,
	\item [(B)] $\local(O(\log^*(n)))=\rlocal(O(\log^*(n)))=\tail=\CONT$,
	\item [(C)] $\BOREL= \BAIRE=\measure= \FIID$,
	\item [(D)] none of above. 
\end{itemize}
Moreover, to what class a given LCL problem $\Pi$ belongs is decidable, albeit $\pspace$-hard. 
\end{corollary}

We do not discuss the classes $\tail, \CONT, \unifCONT$ or block factors here although these classes are present in \cref{cor:main,fig:big_picture_paths}. They are discussed in \cite{grebik_rozhon2021toasts_and_tails} for the more general setup of grids.

%The picture is very clean. 
%We find that the stronger variants of $\borel$, that is $\measure, \FIID, \baire$ still contain the same problems. 
\paragraph{Derandomization Perspective}
The relation of classes $\measure$ and $\FIID$ to $\borel$ is analogous to the relation of randomized algorithms to deterministic algorithms. This means that \cref{thm:main} can be seen as a \emph{derandomization} result. The topic of derandomization is in the center of interest in distributed algorithms \cite{chang2016exp_separation,ghaffari_harris_kuhn2018derandomizing,ghaffari_kuhn_maus2017slocal,ghaffari2018congest-derandomizing,RozhonG19,ghaffari_grunau_rozhon2020improved_network_decomposition} and in complexity theory in general \cite{arora_barak2009computational_complexity_modern_approach,goldreich2008computational_complexity_book}. 
We are not aware of similar derandomization results in the area of descriptive combinatorics, except the case of LCLs without inputs on paths in \cref{sec:no_inputs}. 
For concrete problems, like the famous circle squaring problem, the derandomization of the construction (that is, replacing measurable pieces by borel measurable pieces) was done in the work of \cite{Circle} that improved the previous ``measurable version'' \cite{OlegCircle}. 

In general, it is known that randomness helps if we do not bound the expansion of the graph class under consideration.
As an example we recall that proper vertex $3$-coloring or perfect matching is not in the class $\borel$ for infinite $3$-regular trees \cite{DetMarks}. This implies no nontrivial deterministic local algorithm \cite{Bernshteyn2021LLL}. On the other hand, the two problems are in the class $\measure$ and $\baire$ \cite{BrooksMeas}  and the $3$-coloring problem, in fact, admits a nontrivial randomized local algorithm \cite{ghaffari2020Delta_coloring}.
What if the graph family is of subexponential growth? The celebrated conjecture of Chang and Pettie that the deterministic complexity of the Lovász Local Lemma (LLL) problem is $O(\log n)$\footnote{They conjectured that the randomized complexity is $O(\log\log n)$ which is equivalent to this claim. }  would imply that if the graph class is of subexponential growth, the class of ``LLL-type problems'', the only local class where randomness helps essentially, is not present \cite{chang2016exp_separation,chang2017time_hierarchy}. 
We conjecture that it is a general phenomenon that randomness does not help in graphs of subexponential growth. In particular, we conjecture that our result that $\borel = \measure = \fiid$ on paths holds for all graphs of subexponential growth.
We note that this is not known even for $2$-dimensional grids, see \cite{grebik_rozhon2021toasts_and_tails}.

The classification of \cref{thm:main} is decidable, though, in fact, $\pspace$-hard \cite{balliu2019LCLs_on_paths_decidable}. We think it is an exciting complexity-theoretic problem to understand whether it is in fact $\pspace$-complete \cite{balliu2019LCLs_on_paths_decidable}. The fact that the classification from \cref{cor:main} is decidable corresponds to the fact that there is a reasonable combinatorial classification of problems in each class in \cref{fig:big_picture_paths}. We will discuss the classification of the problems in $\borel$ relevant for \cref{thm:main} in \cref{subsec:intuition_about_main}.

\section{Warm-up: The Case of No Inputs}
\label{sec:no_inputs}

In this section we explain the following folklore wisdom: an LCL without inputs can be constructed on the graph $G_\alpha=(\mathbb{S}^1,E)$, from \cref{sec:intro}, if and only if the local complexity of the problem is $O(\log^* n)$. This follows from a classification of those problems in the distributed world \cite{brandt_grids} combined with the fact that irrational rotations are ergodic.
%We are not aware of a work that explicitly mentioned this connection. 
This should be contrasted with \cref{thm:main} that shows that this is no longer case if one considers the general class of LCLs  (with inputs). 

We will use that every LCL problem can be presented in a so-called normal form (\cref{def:normal_form_line}). In this form, the checking algorithm considers only a pair of neighboring nodes in the cycle. We can restrict the whole discussion just to problems in the normal form. 
If the problem $\Pi$ has no inputs, we can hence assign it an automaton $A = A(\Pi)$ with states from $\Sigma_{out}$ and transition given by the pairs $\sigma_1, \sigma_2 \in \Sigma_{out}$ that are compatible with each other. 
For example, automatons for some problems are in \cref{fig:automatons}. 

\begin{figure}
    \centering
    \includegraphics[width = .8\textwidth]{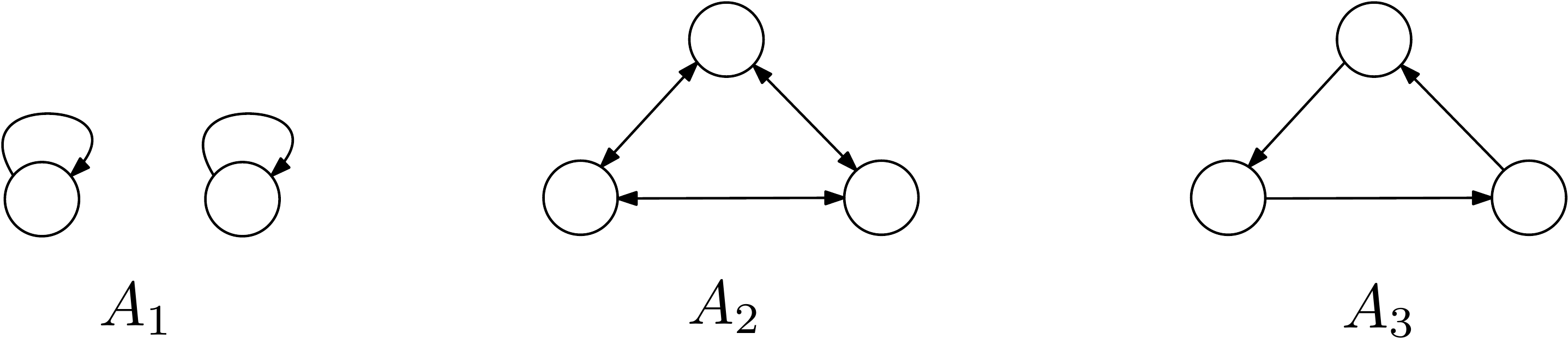}
    \caption{The figure contains automatons for three different problems of local complexities $O(1)$, $\Theta(\log^* n)$, $\Theta(n)$. $A_1$ represents a problem where all nodes are to be labelled with color $1$, or all nodes should be labeled with color $2$. The automaton $A_2$ corresponds to the $3$-coloring problem. The automaton $A_3$ corresponds to the problem \markkth for $k=3$.}
    \label{fig:automatons}
\end{figure}

\paragraph{Distributed Classification}
The classification theorem from \cite{brandt_grids} asserts that the complexity of the LCL problem $\Pi$ can be deduced by looking at its automaton $A$. We will now sketch its proof. It has three parts:
\begin{enumerate}
    \item If the automaton $A$ contains a loop at some node $\sigma$ (\cref{fig:automatons}) $\Pi$ can be solved in $O(1)$ rounds (in fact, $0$ rounds). 
    On the other hand, if $A$ does not contain a loop, any solution of it yields a solution to a problem of coloring the graph with $|\Sigma_{out}| = O(1)$ many colors that is known to require $\Omega(\log^* n)$ rounds \cite{linial1987LOCAL}.  
    \item Suppose that there is $\sigma \in A$ and some $k_0 \ge 0$ such that there is a closed walk from $\sigma$ to $\sigma$ in $A$ of any length $k \ge k_0$. That is, $A$ contains an ergodic component. Such problems can be solved in $O(\log^* n)$ rounds as follows. 
    
    It is known that in $O(\log^* n)$ rounds we may mark a set of points $S$ such that any two consecutive points in $S$ are either $k$ or $k+1$ hops apart. The solution is to label those points by $\sigma$ and extend the labeling between two consecutive points via a walk from $\sigma$ to $\sigma$ of length $k$ or $k+1$ in additional $k+1$ steps. 
    
    \item It is known that oriented graphs without ergodic component are such that each of their strongly connected components is a blowup of a consistently oriented cycle. The automaton consisting of a consistently oriented cycle of length $k$ corresponds to a problem of marking every exactly $k$-th element of the input graph. We call this general problem \markkth. For $k=2$ this is just $2$-coloring and even for general $k$, this problem and its generalization to problems with automatons without ergodic component are easily seen to require $\Omega(n)$ rounds. 
\end{enumerate}

\paragraph{Extension to the Cycle}

We can now consider the problem of labeling the nodes of $\mathbb{S}^1$ with labels from $\Sigma_{out}$ such that each color class is a finite union of intervals. 

On one hand, consider a problem $\Pi$ with an ergodic component. That is, assume there is $\sigma$ and $k_0 \ge 0$ such that for all $k \ge k_0$, there are closed walks from $\sigma$ to $\sigma$ of length $k$ in the respective automaton $A$.  
We claim this problem allows an interval coloring. The reason is that we can first choose an interval of sufficiently small length depending on $\alpha$ and $k_0$. This marks a set $S$ in $G$ with the property that any nodes of $S$ consecutive in $G$ are at least $k_0$ hops apart. On the other hand, there is a constant $\ell = \ell(\alpha, k_0)$ such that they are at most $\ell$ hops apart. 
We can label nodes in $S$ with $\sigma$ and think of $S$ as being split into at most $\ell - k_0 + 1$ classes depending on the distance to the next point in $S$ in $G$. 
This splits $S$ into finitely many intervals. For each of those classes we fill in the labels on the way to the next node in $S$. This labels all nodes in $G$ and we use only finitely many intervals. 

On the other hand, consider a problem $\Pi$ without an ergodic component. We have discussed that those problems are only little bit more general then the class of \markkth problem. But the argument that $2$-coloring cannot be constructed with intervals on $\mathbb{S}^1$ from \cref{sec:intro} carries over to \markkth problem and also to problems without an ergodic component. The general argument that such problems do not admit Borel or Lebesgue measurable solutions exploits, again, the fact the irrational rotations are ergodic. 

The same ideas work to formally show that all non-trivial complexity classes of LCLs without inputs that we consider are in fact equal on oriented paths. %, see \cref{fig_big_picture_paths}.
%We encourage the reader to fill in all missing details to get the following theorem.

%\todo{chybi tu zminene ostatni tridy, ale to je podle me v pohode}

\begin{theorem}
For LCL problems {\bf without input} on oriented lines we have that $\Pi \in \local(O(\log^* n))$ if and only if $\Pi \in \borel$. The latter means that there is a Borel solution to $\Pi$ on the cycle graph $G = (\mathbb{S}^1, E)$. 
\end{theorem}

That is, in \cref{fig:big_picture_paths} without inputs there are only three classes of problems: those admitting a solution where every node uses the same label, then comes the interesting class $\local(O(\log^* n)) = \borel = \dots$ containing basic symmetry breaking problems, and finally there are problems that are variants of the \markkth problem. 
Deciding what is the complexity of a given problem can be done in polynomial time.

\section{Intuition for The Case of Inputs}
\label{subsec:intuition_about_main}

\begin{comment}
Intuitively, there are two reasons why an LCL is not local. Either the LCL is of type exactly-$\ell$-independent set, or some information has to travel an unbounded distance \todo{V: I am not sure if I agree with myself}. 
In Borel solutions, only the first problem stays, while we do not care about potentially unbounded distances. 
\end{comment}

In previous section we observed that for LCLs {\bf without inputs} the class $\local(O(\log^* n))$ is equal to the class $\borel$. 
One should think of this phenomenon as caused by the fact that there are not that many LCL problems without input, not that the computational power of the two models is the same.

The situation for general LCL problems is different.
Intuitively, the computational power for the classes $\local$ does not have access to any global information, while for $\BOREL$ it does.
This can be illustrated by the following example problem.

%\paragraph{A Problem Separating $\borel$ and $\local(O(\log^* n))$}

\begin{definition}[\blockcol, \cref{fig:3coloring_of_blocks}]
The problem \blockcol is defined as follows. 
There is binary input $\Sigma_{in} = \{L, \emptyset\}$ that should be interpreted as inducing consecutive blocks on the input path (cycle). 
The output set is $\Sigma_{out} = \{\tR, \tG, \tB\}$. 
The problem asks to provide a $3$-coloring of the given blocks. 
That is, the solution is checked as follows: $\fP$ accepts $((*,\emptyset) \times (\tA, \tA))$ for $\tA \in \{\tR, \tG, \tB\}$ (the nodes in the same block agree on their color) and $((*, L), (\tA, \tB)$ for $\tA, \tB \in \{\tR, \tG, \tB\}$ and $\tA \not = \tB$ (neighboring blocks have different color). 
\end{definition}

\begin{figure}
    \centering
    \includegraphics[width = .8\textwidth]{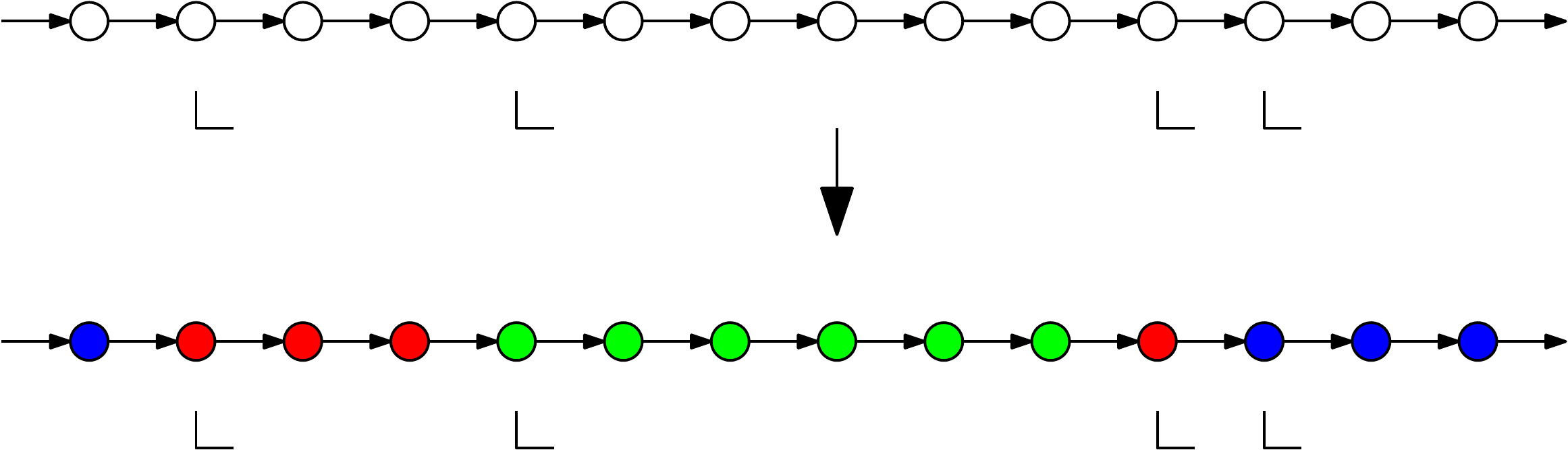}
    \caption{The input and a solution to the problem \blockcol. }
    \label{fig:3coloring_of_blocks}
\end{figure}

Note that since there is no a priori bound on the distance between L's it is not hard to show that \blockcol is not in $\local(O(\log^*(n)))$.
On the other hand, the computation that is allowed in $\BOREL$, see the intuition in \cref{sec:intro}, allows us to work with blocks of unbounded size as if their size was bounded, as well as checking whether input letters $L$ appear at all.
That is \blockcol is in $\borel$ but not in $\local(O(\log^* n))$.

\paragraph{What Problems are not in $\borel$?}

We used the fact that irrational rotations are ergodic to show that $2$-coloring is not in $\borel$.
On a more concrete level, what happens in $2$-coloring is that if a node picks a color, say blue, then all the nodes in the same oriented path already know their color.
Said differently, the decisions of all nodes are completely dependent. 
However, it can be shown that every measurable function needs to exhibit some correlation decay between decisions of nodes that are far away from each other.

Problems that force a full dependency behavior will serve us as a base to show that a given problem is not in $\borel$.
For those problems, it is a bit simpler to think about them as having input on the edges, not vertices.

\begin{definition}[Permutation problem]\label{def:PermProblem}
Each edge is labeled with one permutation $\pi$ from the set $P$. Those permutations are defined on a set $\Sigma_{out}$. The task is to label each node with a label from  $\Sigma_{out}$ such that for consecutive $u, v$ we have $\pi_{uv}(\sigma_u) = \sigma_v$. 
\end{definition}

Note that having an instance of this problem and a node that picks its output label fully determines the decisions of all nodes in his component.
It is easy to see that a permutation problem is solvable in $0$ rounds of $\local$ model if there is a $\sigma \in \Sigma_{out}$ that is a fixed point of all permutations of $P$. If this is not the case, we say that $P$ is \emph{mixing}. 

Our proof of \cref{thm:main} consists of two steps. First, we show that an LCL $\Pi$ allows for a solution with a $\borel$ algorithm unless $\Pi$ is at least as hard as some mixing problem $P$ (in fact, we work with a more general and more technical definition). 
Second, we show that  mixing problems are not contained in any of the classes $\borel, \measure, \fiid, \baire$.

\section{Definitions and Complexity classes}
\label{subsec:def_local}

%What is LCL problem...
%Let $\Pi=(\Sigma_{in},\Sigma_{out},\mathcal{P})$ be an LCL problem.
We use standard graph theoretic and set theoretic notation. The symbol $\triangle$ is used for symmetric difference of two sets. 
Any \emph{graph} we consider is usually either an infinite oriented path or a finite oriented cycle.
A correct solution to an LCL $\Pi$ is a $\Sigma_{in}$-$\Sigma_{out}$-labeling that satisfies the constraints from $\Pi$.
We call such a labeling \emph{$\Pi$-coloring}.

We start by defining the complexity classes from distributed computing, with focus on finite oriented cycles. A local algorithm is the following (see \cite{grebik_rozhon2021toasts_and_tails} for a more formal and more general definition):

\begin{definition}[A deterministic local algorithm (for oriented cycles)]
\label{def:local_algorithm}
Let $\Sigma_{out}$ be a finite set.
A \emph{local algorithm $\mathcal{A}$} is a sequence of pairs $\{\mathcal{A}_n,f(n)\}_{n\in\mathbb{N}}$, where $f(n)\in \mathbb{N}$ and $\mathcal{A}_n$ is a function such that its input is an oriented path of length $2f(n)+1$ (that is, the path has radius $f(n)$ from the perspective of the middle node) where each node has, moreover, a label from some space. The output is a label from $\Sigma_{out}$. 

Applying a local algorithm to a cycle means applying it to the $f(n)$-hop neighborhood of all nodes. 
\end{definition}

\begin{definition}[Complexity classes -- Distributed computing]
We say that $\Pi$ is in the class $\local(f)$ if for every $n\in \mathbb{N}$ there is a deterministic algorithm $\fA_n$ of locality $f(n)$ that takes as input $n$-node cycle labeled with unique identifiers from $[\poly(n)]$ and labeled with labels from $\Sigma_{in}$, and outputs a $\Pi$-coloring.
\end{definition}

Direct bridge between computing on finite cycles and on infinite paths is formed by so-called \emph{oblivious local algorithm}.
This is a variant of the definition of local algorithm that does not know the size of the input graph, therefore, we can run them on finite or infinite graphs without any difficulty. Moreover, the algorithm is randomized: instead of getting a unique identifier, every node has a random string, i.e., is labelled with $2^\mathbb{N}$. The algorithm has to be correct with probability $1$. 
Formally, an oblivious local algorithm is a pair $\fA=(\fC,g)$, where $\fC$ is a measurable collection of rooted $2^\mathbb{N}$-$\Sigma_{in}$-labeled finite oriented paths and $g:\fC\to \Sigma_{out}$ is a measurable function.
An execution of $\fA$ is defined as follows.
Given a graph $G$ that is $2^\mathbb{N}$-$\Sigma_{in}$-labeled and a node $v\in V(G)$, we let $v$ discover its neighborhoods until it encounters element of $\fC$ and then define $\fA(v)$ to be evaluation of $g$ on this neighborhood.
An oblivious local algorithm $\fA$ solves an LCL $\Pi$ if $\fA$ produces a $\Pi$-coloring with probability $1$. 
We say that an oblivious local algorithm $\fA$ is in the class $\tail(f)$, for some $f:\mathbb{N}\to \mathbb{N}$, if
$$\P(R_\fA>r)\le \frac{1}{f(r)},$$
where $\P(R_\fA>r)$ is the supremum over all events where we first fix $\Sigma_{in}$-labeling on rooted oriented path of radius $r$ and then ask what is the probability that the root need to explore more than this neighborhood.

\begin{definition}
We say that an LCL $\Pi$ is in the class $\tail$ if there is an oblivious randomized algorithm $\fA=(\fC,g)$ that solves $\Pi$ and such that $\fA\in \tail(\omega(1))$.
\end{definition}

There is a tight connection between oblivious local algorithms and finitary factors of iid processes, see \cite[Section~2]{grebik_rozhon2021toasts_and_tails} for precise discussion.
Moreover, we have the following theorem:

\begin{theorem}[\cite{grebik_rozhon2021toasts_and_tails}]
Let $\Pi$ be an LCL problem on paths.
Then $\Pi$ is in the class $\tail$ if and only if it is in the class $\local(O(\log^*(n)))$.
\end{theorem}

These classes also coincide with the class $\rlocal(O(\log^*(n)))$, see \cite{chang2016exp_separation}.
Next we define the complexity class $\borel$.
In \cref{sec:intro}, we described two possible approaches to ``Borel algorithms'' (a) model with irrational rotation of the circle (b) countable iteration of MIS and $O(1)$ constructions.
Even though, as we show later, these models are the same and, in the context of oriented paths, capture the full computational power of Borel, we introduce the setting formally to help the reader pass to rather more abstract classes $\baire$ and $\measure$, as well as to the context of different families of graphs.

A \emph{standard Borel space} is a pair $(X,\fB)$ where $X$ is a set and $\fB$ is a $\sigma$-algebra of subsets of $X$ that is equal to a $\sigma$-algebra of Borel sets for some Polish topology $\tau$ on $X$, see \cite[Section~12]{KecClassic}.
That is to say, there is a \emph{Polish topology}\footnote{Every ``nice'' topological space is a Polish space: compact metric spaces, separable Banach spaces, $\mathbb{N}^\mathbb{N}$ etc.}, i.e., separable and completely metrizable, $\tau$ on $X$ such that $\fB=\fB_\tau$, where $\fB_\tau$ are the Borel sets generated by open sets from $\tau$.

As in the case of irrational rotations of the circle we consider graphs that are induced by a map.
Namely, a \emph{Borel autormorphism} $S:X\to X$ is a bijection that is Borel measurable.
We say that $S$ is aperiodic\footnote{An example of aperiodic Borel automorphism is any irrational rotation of the circle.} if $S^n(x)=x$ implies $n=0$ for every $x\in X$.
Define an oriented graph $G_S=(X,E)$, where $(x,y)\in E$ if and only if $y=S(x)$.
It is easy to see that $S$ is aperiodic if and only if connected components of $G_S$ are doubly infinite lines.

\begin{definition}[$\borel$]
We say that $\Pi$ is in the class $\borel$ if for every standard Borel space $(X,\fB)$, aperiodic Borel automorphism $S$ and a Borel map (partition) $i:X\to \Sigma_{in}$ there is a Borel map $f:X\to \Sigma_{out}$ such that the $\Sigma_{in}$-$\Sigma_{out}$-labeling induced by $(i,f)$ is a $\Pi$-coloring of the graph $G_S$.
\end{definition}

\subsection{Descriptive classes}

A natural way how to relax the requirement on a function to be Borel measurable is to use probability measures.
Namely, in the example with circle we could have asked for $2$-coloring of $G_\alpha$ that is defined merely almost everywhere (with respect to Lebesgue measure), or equivalently that is measurable with respect to the $\sigma$-algebra of all Lebesgue measurable sets (the measurable completion of Borel sets).
In the context of inifnite paths, these concepts are the same.
However, even in the proof that there is no Borel $2$-coloring of $G_\alpha$ we actually proved that there is no Lebesgue measurable such coloring using the ergodicity of the rotation.

We say that $(X,\fB,\mu)$ is a \emph{standard probability space} if $(X,\fB)$ is a standard Borel space and $\mu$ is a Borel probability measure defined on $\mathcal{B}$, see \cite[Section~17]{KecClassic}.
Usually we assume that the $\sigma$-algebra $\fB$ is understood from the context and write just $(X,\mu)$.
Note that if $X$ is uncountable, then the space of all Borel probability measures on  $(X,\fB)$ is also uncountable.\footnote{In fact, it is naturally a standard Borel space.}

\begin{definition}
We say that $\Pi$ is in the class $\measure$ if for every standard Borel space $(X,\fB)$, aperiodic Borel automorphism $S$, Borel probability measure $\mu$ on $(X,\fB)$ and a Borel map (partition) $i:X\to \Sigma_{in}$ there is a Borel map $f:X\to \Sigma_{out}$ such that the $\Sigma_{in}$-$\Sigma_{out}$-labeling induced by $(i,f)$ is a $\Pi$-coloring of the graph $G_S$ $\mu$-almost everywhere.
\end{definition}

A topological analogue of the measurable relaxation is given in terms of meager, resp. comeager, sets.
Recall that if $(X,\tau)$ is a topological space, then a $\tau$-closed set $Z$ is \emph{nowhere dense} if it does not contain any open set.
A set $M\subseteq X$ is called \emph{$\tau$-meager} if it is a subset of a countable union of closed nowhere dense sets.
A complement of a $\tau$-meager set is a $\tau$-comeager set.
When the context is clear we not mention the topology and say e.g. comeager.
The intuition is that meager sets are topoogically small.
Indeed, if $(X,\tau)$ is a Polish space, then every comeager set is dense and, in particular, non-empty, see \cite[Section~16]{KecClassic}.

We defined a standard Borel space as a pair $(X,\mathcal{B})$, where $\mathcal{B}$ is a $\sigma$-algebra of Borel sets for some Polish topology $\tau$ on $X$.
However, once $(X,\fB)$ is fixed there are many Polish topologies for which $\fB$ is the Borel $\sigma$-algebra.
For example consider the interval $[0,1]$ with the standard topology and with the topology that makes $\{1\}$ isolated point (while keeps the topology on $[0,1)$).
These topologies are both Polish and give the same Borel $\sigma$-algebra, in fact every Borel set is clopen in some finer Polish topology that gives the same Borel sets, see \cite[Section~13]{KecClassic}.
Note that while in the standard topology $\{1\}\subseteq [0,1]$ is meager (closed nowhere dense), in the other topology it is not meager.

\begin{definition}
We say that $\Pi$ is in the class $\baire$ if for every standard Borel space $(X,\fB)$, aperiodic Borel automorphism $S$, Polish topology $\tau$ whose Borel $\sigma$-algebra is $\fB$ and a Borel map (partition) $i:X\to \Sigma_{in}$ there is a Borel map $f:X\to \Sigma_{out}$ such that the $\Sigma_{in}$-$\Sigma_{out}$-labeling induced by $(i,f)$ is a $\Pi$-coloring of the graph $G_S$ on a $\tau$-comeager set.
\end{definition}

%We conclude this subsection with a trivial observation.
%that is valid in a bigger generality.

\begin{claim}
Let $\Pi$ be an LCL and suppose that $\Pi$ is in the class $\borel$.
Then $\Pi$ is in the classes $\baire$ and $\measure$.
\end{claim}

Last class that should be mentioned to make the picture complete is the class $\CONT$.
LCLs that admit so-called \emph{continuous} solution were extensively studied in \cite{GJKS,Bernshteyn2021local=cont}.
Ultimately, it was proved independently by Bernshteyn \cite{Bernshteyn2021local=cont} and Seward \cite{Seward_personal} that $\local(O(\log^*(n)))=\CONT$ for LCLs without inputs on Cayely graphs of finitely generated groups.
We refer the reader to \cite[Section~7.1]{grebik_rozhon2021toasts_and_tails} for a high-level explanation of the class $\CONT$ and a sketch of the proof that it is equal to $\local(O(\log^*(n)))$ for LCLs with inputs on $\mathbb{Z}^d$.

\subsection{Shifts}\label{subsec:shifts}

An important examples of Borel automorphism from the perspective of descriptive combinatorics and random processes are the Bernoulli shifts.
Since examples that show that given LCL is hard are of this type, we properly define all the involved notions.
In general, the main difficulty is to decide how to involve inputs in the definition.
Also, we remark that there is a slight change of perspective:
Before (a) nodes were elements of some set and edges were oriented pairs of nodes (b) input labeling and solution of a given LCL were maps that color nodes.
For shifts (a) nodes are infinite oriented paths with labeled vertices and, abstractly, two such graphs form an edge if one is a shift of the other (b) solution of a given LCL is a map to a space of solutions that is invariant under the shift-action.

Let $X$ be a set and consider the set of doubly infinite oriented lines labeled with $X$, i.e., $X^\mathbb{Z}$.
We denote by $\cdot$ the natural shift action $\mathbb{Z}\curvearrowright X^\mathbb{Z}$, that is
\begin{equation}
    (n\cdot x)(k)=x(n+k)
\end{equation}
for every $n,k\in \mathbb{Z}$ and $x\in X^\mathbb{Z}$.
Given two such spaces $X^\mathbb{Z}$ and $Y^\mathbb{Z}$ together with a map $f:X^\mathbb{Z}\to Y^{\mathbb{Z}}$, we say that $f$ is \emph{equivariant} if 
$$n\cdot f(x)=f(n\cdot x)$$
for every $n\in \mathbb{Z}$ and $x\in X^{\mathbb{Z}}$.

In general, there are two types of spaces of this form, moreover, endowed with natural topology, that we consider in this paper.
First type are the spaces of inputs or outputs and their subspaces, e.g., $\Sigma_{in}^{\mathbb{Z}}$ and $\Sigma_{out}^{\mathbb{Z}}$.
Second type are the spaces for breaking symmetries in determinstic and randomized setting, e.g., $(2^{\mathbb{N}})^{\mathbb{Z}}$, i.e., paths labeled with infinite strings of $0$s and $1$s.
The spaces $\Sigma_{in}$ and $\Sigma_{out}$ are endowed with the discrete topology, $2^\mathbb{N}$ with its standard compact metrizable topology and all the other spaces with the product topology.
It is a standard fact that the spaces mentioned above are compact and all the spaces we create from them will be Polish spaces.
When we say Borel map we always refer to the Borel $\sigma$-algebra coming from this topology.

Given an LCL $\Pi$, we define $X_\Pi\subseteq \Sigma_{in}^{\mathbb{Z}}\times \Sigma_{out}^{\mathbb{Z}}$ to be the space of all $\Pi$-colorings, that is, $x\in X_{\Pi}$ if $x$, viewed as $\Sigma_{in}$-$\Sigma_{out}$-labeling of infinite oriented path, is a $\Pi$-coloring.
It is easy to see that $X_\Pi$ is closed under the shift action.

\begin{definition}
We say that $\Pi$ is in the class $\FIID$ if there is an equivariant Borel map $f:\Sigma_{in}^\mathbb{Z}\times (2^{\mathbb{N}})^\mathbb{Z}\to \Sigma_{out}^{\mathbb{Z}}$ such that for every fixed $\fI\in \Sigma_{in}^\mathbb{Z}$ we have that 
$$(\fI,f(\fI,x)) \in X_{\Pi}$$
holds with probability $1$ with respect to the uniform measure on $(2^{\mathbb{N}})^\mathbb{Z}$.
\end{definition}

Note that if there were no inputs, then $\Pi$ is in the class $\FIID$ if and only if $\Pi$ admits so-called \emph{factor of iid process (fiid)} solution.
Also in that situation fiid is just a special case of a measurable solution, that is $\measure\subseteq \fiid$ without inputs.
In the case with inputs we can only say that $\borel\subseteq \FIID$.

\section{Proof of \cref{thm:main}}

\subsection{Normal Form}
\label{subsec:lcl_paths_normal_form}

As a usual first step, we show that every problem is equivalent to a nice problem.
Informally, LCL $\Pi$ is in the \emph{normal form} if its locality is, so to say, $1.5$.
Namely, given a candidate $\Sigma_{in}$-$\Sigma_{out}$-labeling every node needs to check its and their right neighbor inputs and output label (in fact checking just neighbors output is enough) to decide whether the candidate labeling is a $\Pi$-coloring.   
Since the definition of normal form is crucial for everything that follows, we state it formally.

%An \emph{LCL with input $\Sigma$} is a collection $\Pi$ that assigns to $A\in \Sigma$ the set $\tau(A)$ of ordered triplets from some set $b$ that are indexed by $(-,0,+)$.
%The space of $\Pi$-solutions is a subspace $X_\Pi\subseteq \Sigma^\mathbb{Z}\times b^\mathbb{Z}$, where $(x,y)\in X_\Pi$ if for every $i\in \mathbb{Z}$ we have $(y(i-1)),y(i),y(i+1)$ is one of the triplet that is assigned to $x(i)$.

\begin{definition}[Normal form of an LCL]
\label{def:normal_form_line}
The \emph{normal form} of an LCL is a triplet $\Pi = (\Sigma_{in}, \Sigma_{out}, \fP)$, where $\Sigma_{in}$ is the (finite) input alphabet, $\Sigma_{out}$ is (finite) output alphabet and $\fP : \Sigma_{in}^2 \times \Sigma_{out}^2 \rightarrow \{true, false\}$.
\end{definition}

\begin{fact}[Lemma 2 in \cite{balliu2019LCLs_on_paths_decidable}]
Every LCL $\Pi$ can be transformed into an LCL $\Pi'$ in the normal form such that the solution of $\Pi$ can be transformed to a solution of $\Pi'$ with a local algorithm with $O(1)$ locality and vice versa. 
\end{fact}

Hence, from now on in this section we consider only LCLs in the normal form.

\subsection{Blocks and Mixing}

In this subsection we introduce several definitions, ultimately yielding the combinatorial condition that characterizes when $\Pi$ is in $\borel$.
The proof of this fact is presented in subsequent subsections.
Recall that we assume that $\Pi$ is as in \cref{def:normal_form_line}.

%We will present a combinatorial property of $\Pi$ that determines whether there is borel/baire/measure solution.

\begin{figure}
    \centering
    \includegraphics[width = .9\textwidth]{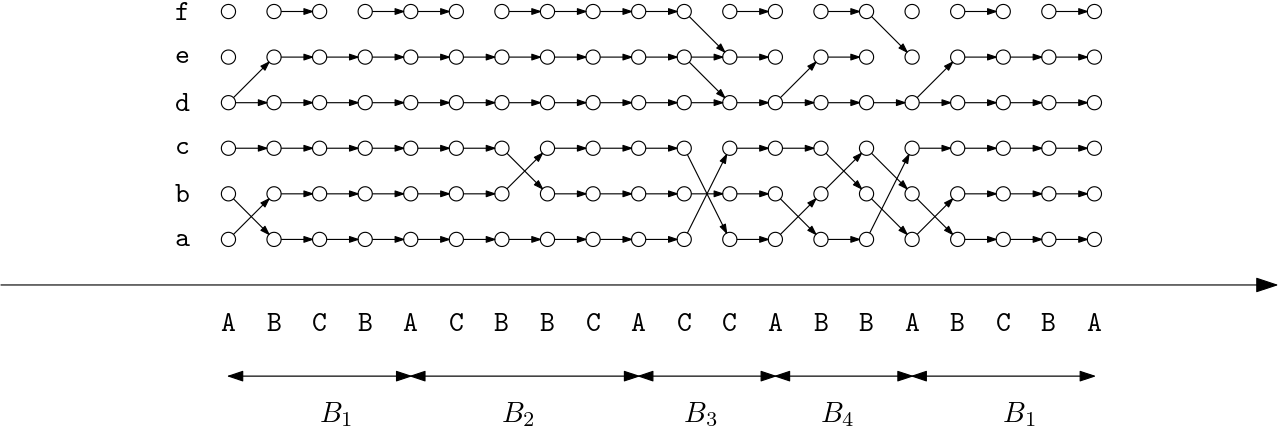}
    \caption{The input seen as a sequence of $\tA$-blocks. In this particular example, we have $\type(B_1) = \type(\tA\tB\tC\tB\tA) = \type(\tA\tB\tC\tB(\tC\tB)^*\tA)$, that is, pumping up the original block with the string $\tC\tB$ does not change the block's type. }
    \label{fig:lines_blocks}
\end{figure}

\begin{definition}[Input graph]
Let $\fI \in \Sigma_{in}^{\mathbb{Z}}$ be an instance of input labeling.
We define an oriented graph $G_{\fI}$, where $V(G_\fI)=\mathbb{Z}\times \Sigma_{out}$ and $((i,\ta),(i+1,\tb))$ forms an oriented edge if $\fP((\fI(i), \fI(i+1)), (\ta, \tb)) = true$.
\end{definition}

The $i$-th layer of the graph $G_{\fI}$, that is, vertices of the form $(i, \tau)$ is denoted as the \emph{slice} $G[i]$. 
The edges between consecutive layers may also be seen as a relation parametrized by the respective two input letters. 
It is easy to see that any doubly infinite orinted path in $\Sigma_\fI$ corresponds to a solution of $\Pi$ and vice versa.

\begin{definition}[$\sigma$-block]
Let $\sigma\in \Sigma_{in}$.
A \emph{$\sigma$-block} $B$ of length $t$ is a sequence $\sigma = \sigma_1, \sigma_2, \dots, \sigma_{t} = \sigma$ with each $\sigma_i \in \Sigma_{in}$. 
\end{definition}

A \emph{block} is a $\sigma$-block for some $\sigma$.
Equivalently, we may see a block as a finite cut in the input graph $G_\fI$ that starts and ends with the same input label, i.e., it determines the input graph on $B$.
For the next definitions we tacitly assume that a given block $B$ of length $t$ is part of some fixed input labeling $\fI\in \Sigma_{in}^\mathbb{Z}$ such that the first letter is at position $0$ and the last at position $t$.

\begin{definition}[Type of a block]%\todo{maybe should be named relation}
The \emph{type} of a block $B$, denoted as $\type(B)$, is the relation $\type(B) \subseteq \Sigma_{out}^2$ such that $(\ta, \tb) \in \type(B)$ for $\ta, \tb \in \Sigma_{out}$ if there is a path in ($G_\fI$ restricted to) $B$ connecting $(0, \ta)$ with $(t, \tb)$.
\end{definition}

A \emph{subpartition (of $\Sigma_{out}$)} is a pair $(\Upsilon,\nabla)$, where $\nabla$ is an equivalence relation defined on some subset $\dom(\nabla)$ of $\Sigma_{out}$ and $\Upsilon=(\mathbb{P},\preceq)$ is a poset where $\mathbb{P}$ is an equivalence relation with $\dom(\mathbb{P})=\Sigma_{out}$ that is coarser than $\nabla$ on $\dom(\nabla)$.
We write $[\ta]_\nabla$ for the equivalence class of $\ta\in \dom(\nabla)$ and $\mathbb{P}(\ta)$ for the equivalence class of $\ta\in \dom(\mathbb{P})$.

\begin{definition}[A permutation block on a subpartition]\label{def:subpartition}
Let $B$ be a block and $(\Upsilon,\nabla)$ a subpartition.
We say that $B$ is a \emph{permutation block on the subpartition $(\Upsilon,\nabla)$} if there is a permutation $\pi_B$ of $\nabla$ such that
\begin{enumerate}
    \item $\pi_B$ is $\mathbb{P}$-invariant, that is, if $\pi_B([\ta]_\nabla)=[\tb]_\nabla$, then $\mathbb{P}(\ta)=\mathbb{P}(\tb)$,
    \item $(\ta, \tb) \in \type(B)$ implies $\mathbb{P}(\ta) \preceq \mathbb{P}(\tb)$,
    \item $\ta\not\in \dom(\nabla)$ and $(\ta, \tb) \in \type(B)$ implies $\mathbb{P}(\ta)\not=\mathbb{P}(\tb)$,
    \item $\ta,\tb\in \dom(\nabla)$, $(\ta, \tb) \in \type(B)$ and $\mathbb{P}(\ta)=\mathbb{P}(\tb)$ implies $\pi_B([\ta]_\nabla)=[\tb]_\nabla$.
\end{enumerate}
holds for every $\ta,\tb\in \Sigma_{out}$.
\end{definition}

\begin{definition}[Permutation group on $(\Upsilon,\nabla)$]
Let $\sigma\in \Sigma_{in}$ and $(\Upsilon,\nabla)$ be a subpartition.
We set $\Gamma_{\sigma, (\Upsilon,\nabla)}$ to be the set of all permutations $\pi_B$ that are induced by permutation $\sigma$-blocks $B$ on the subpartition $(\Upsilon,\nabla)$.
\end{definition}

Observe that permutation $\sigma$-blocks on a subpartition $(\Upsilon,\nabla)$ are closed under concatenation.
Moreover, the operation concatenation of blocks turns $\Gamma_{\sigma, (\Upsilon,\nabla)}$ naturally into a semigroup that acts on $\nabla$.
Since the action is by permutations and $\nabla$ is finite it must be the case that $\Gamma_{\sigma, (\Upsilon,\nabla)}$ is actually a group.

Finally, we are able to define the combinatorial condition of being \emph{mixing} that characterizes hardness of a given problem.

\begin{definition}[Mixing]
Let $\Pi$ be LCL, $\sigma\in \Sigma_{in}$ and $(\Upsilon,\nabla)$ be a subpartition of $\Sigma_{out}$.
We say that $\Pi$ is \emph{$(\sigma, \Upsilon,\nabla)$-mixing} if the canonical action of $\Gamma_{\sigma, (\Upsilon,\nabla)}$ on $\nabla$ does not have a fixed point.
We say that $\Pi$ is \emph{mixing} if there are $\sigma$ and $(\Upsilon,\nabla)$ such that $\Pi$ is $(\sigma,\Upsilon,\nabla)$-mixing.
\end{definition}

Note that every block $B$ is a permutation block on the subpartition $(\Upsilon,\nabla)$, where $\mathbb{P}=\nabla=\{\Sigma_{out}\}$.
On the other hand, there is a permutation block on the subpartition $(\Upsilon,\nabla)$, where $\mathbb{P}=\{\Sigma_{out}\}$ and $\nabla=\emptyset$, if and only if there is an input labeling $\fI\in \Sigma_{in}^\mathbb{Z}$ that does not admit a  $\Pi$-solution, i.e., $(\fI,f)$ is not a $\Pi$-coloring for every $f\in \Sigma_{out}^\mathbb{Z}$.
The proof of this fact is an easy application of K\" onig's Theorem.
In that case we consider, formally, $\Pi$ to be mixing.

\begin{comment}
Whenever there is an instance $\fI$ that does not admit any solution, then $\Pi$ is mixing from the following trivial reason:
Consider a label $\sigma\in \Sigma_{in}$ that appears infinitely often in the $+\infty$ direction and let $i\in \mathbb{Z}$ be such that $\fI(i)=\sigma$.
It follows from K\" onig's Theorem (recall that $\Sigma_{out}$ is finite) that there is $\ell>0$ such that $\fI(i+\ell)=\sigma$ and, when we write $B$ for the restriction of $G_\fI$ to $[i,\dots,i+\ell]$, that $\type(B)=\emptyset$.
Set $\nabla=\emptyset$ and $\mathbb{P}=\{\Sigma_{out}\}$.
It is easy to see that $B$ is a permutation $\sigma$-block on the subpartition $(\Upsilon,\nabla)$.
Consequently, $|\Gamma_{\sigma,(\Upsilon,\nabla)}|=1$ and the action $\Gamma_{\sigma,(\Upsilon,\nabla)}\curvearrowright \emptyset$ does not have a fix point.
\end{comment}

\begin{figure}
    \centering
    \includegraphics[width=.98\textwidth]{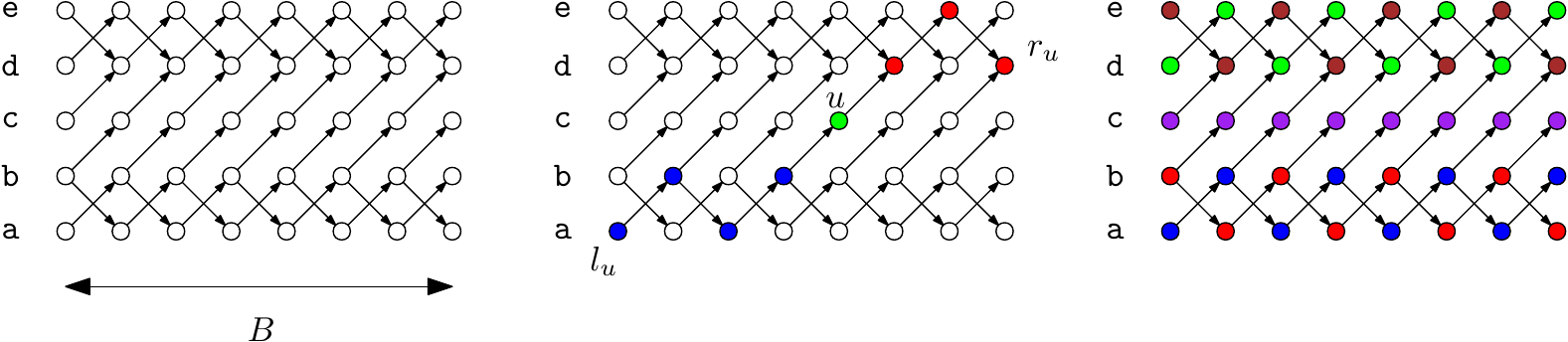}
    \caption{We illustrate several definition on an example graph that could be in fact generated even without input labels. \\
    Left: an example of a permutation block $B$ for $\nabla = \{ \{\ta\}, \{\tb\}, \{\td\}, \{\te\} \}$, $\mathbb{P} = \{ \{ \ta, \tb\}, \{ \td, \te\}, \{\tc\} \}$ and $\{ \ta, \tb\}\preceq \{\tc\} \preceq \{ \td, \te \} $, and $\type(B) = \{(\ta,\td), (\tb, \te),(\ta, \tc),(\tb, \tc),(\tc, \td),(\tc, \te)\}$. We have $\pi_B: \ta\rightarrow\tb, \tb\rightarrow\ta, \td\rightarrow\te, \te\rightarrow\td$. One can check that the block $B$ is a permutation block according to \cref{def:subpartition}. \\ Middle: definition of $r_{i,\tc}$ (red color) and $l_{i,\tc}$ (blue color), where $u=(i,\tc)$ (green color), from \cref{subsec:nonmixing_upper_bound}.  \\
    Right: the equivalence $\Re$ from \cref{subsec:nonmixing_upper_bound} has five classes in our example. }
    \label{fig:perm_block}
\end{figure}

\subsection{Non-Mixing implies $\borel$}
\label{subsec:nonmixing_upper_bound}

We show that if $\Pi$ is not mixing, then it is in the class $\borel$.
This immediately yields that $\Pi$ is in $\measure$, $\baire$ and $\fiid$ by the discussion in \cref{subsec:def_local}.
Recall that we assume that not mixing implies that every instance $\fI\in \Sigma_{in}^\mathbb{Z}$ admits a solution, i.e., there is a doubly infinite oriented path in the input graph $G_\fI$ for every $\fI\in \Sigma_{in}^\mathbb{Z}$.

The high-level strategy is to, first, use the power of Borel to compute the ``reachable'' relation $\Re$ on $G_\fI$, i.e., two nodes are equivalent if they can ``see'' the same infinite rays.
There are two types of equivalence classes, those that contain doubly infinite oriented path and those that do not.
For the sake of brevity lets assume that all equivalence classes are of the first type, we also skip several technical details that appear in the actual proof.
Define an order $\triangleleft$ on $\Re$, where $P\triangleleft P'$ if there is a path in $G_\fI$ from $P$ to $P'$.
It can be shown that $(\Re,\triangleleft)$ is a poset.
Every index can compute $(\Re,\triangleleft)$ in the Borel setting, however, it is not possible to distinguish equivalence classes that ``look the same''.
To overcome this problem we set $(\tilde{\mathbb{P}},\tilde{\preceq})$ to be a poset, where $\tilde{\mathbb{P}}=\Re/\operatorname{Aut}((\Re,\triangleleft))$ and $\tilde{\preceq}$ is induced by $\triangleleft$.
Given an index $i\in \mathbb{Z}$ we define $\nabla_i$ to be the intersection of $\Re$ and $G_\fI[i]$, $\mathbb{P}_i$ to be the intersection of $\tilde{\mathbb{P}}$ and $G_\fI[i]$ indexed by elements from $\tilde{\mathbb{P}}$, and $\Upsilon_i=(\mathbb{P}_i,\preceq_i)$ to be the poset structure on $\mathbb{P}_i$ induced by $\tilde{\preceq}_i$.
Finally, pick a (infinite) set of indices $C$ such that $\fI$ is equal to some $\sigma$ on $C$ and $(\Upsilon,\nabla)=(\Upsilon_i,\nabla_i)$ is the same for every $i\in C$.
The blocks induced by consecutive elements from $C$ are permutation $\sigma$-block on a subpartition $(\Upsilon,\nabla)$.
Since $\Pi$ is not mixing there is a fixed point of the action $\Gamma_{\sigma,(\Upsilon,\nabla)}$.
This allows to pick in a Borel fashion one equivalence class from the ``reachable relation''.
Once this is done it is easy to construct a doubly infinite path within this equivalence class.
Formal proof follows.

\newcommand{\GG}{G_\fI}

\begin{theorem}
\label{cl:paths_upper_bound}
Let $\Pi$ be LCL that is not mixing.
Then $\Pi \in \borel$. 
\end{theorem}
\begin{proof}
Suppose that $(X,\fB)$ is a standard Borel space, $S$ is an aperiodic Borel automorphism and $\iota:X\to \Sigma_{in}$ is a Borel map.
By a result of Kechris et al. \cite{KST}, it is possible to find a sequence of Borel MISs in power graphs of $G_S$ with the power parameter going to infinity.
Then it is routine to verify that all the forthcoming constructions are Borel.
This is because they are local modulo having access to these MISs (even to all at once).

Fix $\fI\in \Sigma_{in}^{\mathbb{Z}}$ that appears in $G_S$ with labels given by $\iota:X\to \Sigma_{in}$, i.e., there is $x\in X$ such that its orbit is isomorphic to $\fI$.
We abuse the notation and will refer to $S^{i}(x)$ simply as to index $i$.
We warn the reader that even though it might look like we have picked $x$ as our reference point (which is certainly not possible in Borel setting) all the constructions are only, implicitly, using the fixed sequence of MISs.

We are now going to describe a sequence of operations with the graph $G_\fI$ that end up in picking a doubly oriented path in it in a Borel fashion.
Several times during the construction we run into a situation that allows us to pick a reference point, i.e., all elements $i\in \mathbb{Z}$ can agree on one index in a Borel way.
It is a standard fact that once this happen it is easy to find a doubly infinite path in a Borel way (recall that we assume that every instance $\fI$ admits at least one solution).

\paragraph{Splitting the input graph $\GG$}

We define an equivalence relation $\Re$ on $\GG$ with two types of equivalence classes.
An equivalence class of the first type, the type that corresponds to the $\nabla$ classes in the definition of a subpartition, is called \emph{intertwined path}, where $P$ is an intertwined path if it is a subset of $\GG$ that contains a doubly infinite oriented path and satisfies \cref{cl:InfSupInLine} below.
One should think of intertwined paths in $\GG$ as analogues of strongly connected components in directed graphs.
The second type of equivalence classes contains for example vertices that are not contained in any doubly infinite oriented path in $\GG$. These vertices are just technical artefacts that, unfortunately, cannot be avoided.

Let $(i,\ta)\in \GG$ and define $r_{i,\ta}$ to be the set of all nodes in $\GG$ to the right from $(i,\ta)$, i.e, nodes that are reachable from $(i,\ta)$ by a directed path to the right.
Similarly, define $l_{i,\ta}$ to be the set of nodes reachable from $(i,\ta)$ to the left.
Say that $(i,\ta)$ and $(i',\ta')$ are $\Re^+$ equivalent, if $r_{i,\ta}\triangle r_{i',\ta'}$ is finite.
It is easy to see that $\Re^+$ is an equivalence relation.
Similarly, define $\Re^-$ using $l_{i,\ta}$.

We claim that the number of $\Re^+$ and $\Re^-$ equivalence classes are bounded by $2^{|\Sigma_{out}|}$ and, hence, are finite.
To see this, suppose otherwise and pick representatives $(i_1,\ta_1),\dots, (i_t,\ta_t)\in \GG$, where $t = 2^{|\Sigma_{out}|} + 1$, of different $\Re^+$ equivalence classes.
Take an index $i\in\mathbb{Z}$ such that $i_k<i$ for every $k\in [t]$ and consider the sets $T_k=\GG[i] \cap r_{i_k,\ta_k}$ for every $k\in [t]$.
By the pigeonhole principle, there are $k_1 \not= k_2$ such that $T_{k_1}=T_{k_2}$.
This implies that $r_{i_{k_1},\ta_{k_1}}\triangle r_{i_{k_2},\ta_{k_2}}$ is finite and, hence, $(i_{k_1},\ta_{k_1})$ and $(i_{k_2},\ta_{k_2})$ are $\Re^+$ equivalent, a contradiction.
Similar argument works for $\Re^-$.

Define the equivalence relation $\Re$ as the intersection of $\Re^+$ and $\Re^-$.
Note that there are only finitely many $\Re$ equivalence classes.

\begin{claim}\label{cl:PathSubset}
Let $P\in \Re$ and $(i,\ta),(i',\ta')\in P$.
Suppose that there is a directed path $p\subseteq \GG$ from $(i,\ta)$ to $(i',\ta')$, in particular, $i\le i'$.
Then $p\subseteq P$.
\end{claim}
\begin{proof}
Let $(j,\tb)\in p$.
Then we have $r_{i',\ta'}\subseteq r_{j,\tb}\subseteq r_{i,\ta}$ and consequently $(j,\tb)$ is $\Re^+$ equivalent to, e.g., $(i,\ta)$.
The case with $\Re^-$ is the same.
\end{proof}

Suppose now that there is an equivalence class $P\in \Re$ that does not go to infinity in both directions, i.e., $P\cap \bigcup_{i\ge 0} \GG[i]$ or $P\cap \bigcup_{i\le 0} \GG[i]$ is finite.
If this happens then it is easy to see that this allows to pick a reference point, i.e., lexicographically minimal/maximal element in such a class (depending in what direction it does not go to infinite).
Therefore we may assume that each $P$ is infinite in both directions. This does not mean that it contains an infinite oriented path though. 
%We strengthen this observation in the following claim. 
%Similarly, suppose that there is $P\in \Re$ that contains

\begin{claim}\label{cl:InfSupInLine}
Let $P\in \Re$ contain an infinite oriented path (in any direction).
Then $P$ contains a doubly infinite oriented path and
$$(\bigcup_{j\ge i} G_\fI[j]\cap P)\setminus r_{i,\ta} \ \operatorname{and} \ (\bigcup_{j\le i} G_\fI[j]\cap P)\setminus l_{i,\ta}$$
are finite for every $(i,\ta)\in P$.
\end{claim}
\begin{proof}
{\bf (a)}.
Let $p\subseteq P$ be a directed infinite path, say, directed to $+\infty$ that starts at $(i,\ta)\in P$.
We show that for every $k\in \mathbb{N}$ there is an infinite directed path $p_k\subseteq P$ that is directed to $+\infty$ and that satisfies $p_k\cap \GG[-k]\not=\emptyset$.
Then a compactness argument implies that there is a doubly infinite oriented path in $P$ since $\Sigma_{out}$ is finite.

Let $k\in \mathbb{N}$ and $(j,\tb)\in P$ be such that $j\le -k$ (this exists by the assumption that elements of $\Re$ are infinite in both directions).
By the definition we have that $(i,\ta)$ and $(j,\tb)$ are $\Re^+$ equivalent.
Since $p\subseteq r_{i\ta}$ is infinite, there must be directed path $q$ that starts at $(j,\tb)$ and joins $p$.
Let $q'$ be a concatenation $q$ and $p$ from the point where they intersect.
Note that $q\cap \GG[-k]\not=\emptyset$ and $q'\subseteq P$ by \cref{cl:PathSubset}.
This proves first part of the claim.

{\bf (b)}.
Let $(i,\ta)\in P$ and suppose that $p\subseteq P$ is a doubly infinite directed path.
Suppose, e.g., that $(\bigcup_{j\le i} G_\fI[j]\cap P)\setminus l_{i,\ta}$ is infinite, the other case is analogous.
Let $\{(j_k,\tb_k)\}_{k\in \mathbb{N}}$ be a witness to this fact and suppose that $j_k<i$.
Since $(i,\ta)$ and $(j_k,\tb_k)$ are $\Re^+$ equivalent and $p\cap r_{i\ta}$ is infinite we have, as above, that there is a path $q_k$ directed to the right that starts at $(j_k,\tb_k)$ and joins $p$.
Let $(i,\ta_k)$ be the point where $q_k$ intersects $\GG[i]$.
Note that by \cref{cl:PathSubset} we have that $(i,\ta_k)\in P$.
Since $\GG[i]$ is finite, we may assume by the pigeonhole principle that $\ta_k=\ta'$.
We have $(i,\ta')\in P$, in particular, $(i,\ta)$ and $(i,\ta')$ are $\Re^-$ equivalent.
However, $\{j_k,\ta_k\}_{k\in \mathbb{N}}\subseteq l_{i,\ta'}$ and that is a contradiction with $\{j_k,\ta_k\}_{k\in \mathbb{N}}\subseteq (\bigcup_{j\le i} G_\fI[j]\cap P)\setminus l_{i,\ta}$.
The proof is finished.
\end{proof}

Define $\Re^{\leftrightarrow}$ to be the set of those equivalence classes from $\Re$ that contain doubly infinite oriented path.

\paragraph{Poset $\tilde{\Upsilon}=(\tilde{\mathbb{P}},\tilde{\preceq})$}

We show that there is a quasiorder structure $\triangleleft$ on $\Re$ that is a poset structure when restricted to $\Re^{\leftrightarrow}$.
Let $P,P'\in \Re$ and suppose that there is a directed path in $\GG$ from $P$ to $P'$.
If there are not infinitely many such paths in both directions of $\mathbb{Z}$, then we can pick the lexicographically minimal/maximal (depending on in what direction is the condition violated) such path (for all pairs of $P,P'\in \Re$ that violate this condition) and that would give a reference point picked in a Borel way.
Therefore we assume that this condition is satisfied.
Set $\triangleleft$ to be the transitive closure of the relation that relates $P,P'\in \Re$ if there is a directed path from $P$ to $P'$ in $\GG$.

\begin{claim}\label{cl:Poset}
Let $P\in \Re^{\leftrightarrow}$ and $Q\not =Q'\in \Re$ be such that $Q\triangleleft P\triangleleft Q'$.
Then $Q'\not\triangleleft Q$.
In particular, $\triangleleft$ is a poset when restricted to those $P\in \Re^{\leftrightarrow}$.
\end{claim}
\begin{proof}
Suppose that $Q'\triangleleft Q$.
Pick $(i',\ta')\in Q$ and a directed path $p\subseteq \GG$ that starts at $(i',\ta')$ and ends at $(i,\ta)\in P$.
By \cref{cl:InfSupInLine} we have that $r_{i,\ta}\cap P$ is infinite.
Consequently, $r_{i',\ta'}\cap P$ is infinite and therefore there is a directed path from every $(j,\tb)\in Q$ to $P$ in $\GG$.

Pick now any directed path $q$ from $Q'$ to $Q$.
This path can be extended to a path $q'$ that goes from $Q'$ to $P$ through $Q$ by the previous paragraph.
Let $(i,\ta)\in P$ be its endpoint.
By \cref{cl:InfSupInLine}, $r_{i,\ta}\cap P$ contains $P\cap \bigcup_{j\ge i'}\GG[j]$ for some $i\le i'$.
Let $p$ be any path that connects $P$ and $Q'$ and has a starting point at some element from $P\cap \bigcup_{j\ge i'}\GG[j]$.
This yields a path from $Q'$ to $Q'$ through $Q$ and $P$.
By \cref{cl:PathSubset} we get that $Q'=P=Q$.
\end{proof}

Let $\Re/\triangleleft$ be the factor given by the quasi order, i.e., $(\Re/\triangleleft, \triangleleft)$ is a poset.
We think of each class in $\Re/\triangleleft$ as the corresponding unions of the classes that are glued together by $\triangleleft$.
That is to say, $\Re/\triangleleft$ is an equivalence relation on $\GG$ that is coarser than $\Re$.
Note that by \cref{cl:Poset} we have that $\Re$ and $\Re/\triangleleft$ coincide on $\bigcup_{P\in \Re^{\leftrightarrow}}P$, in fact, every $P\in \Re^{\leftrightarrow}$ is not glued together with any other element from $\Re$.
Let $\operatorname{Aut}^{\leftrightarrow}(\Re/\triangleleft,\triangleleft)$ denotes the group of all automorphisms that preserves set-wise $\Re^{\leftrightarrow}$.
Define
$$\tilde{\mathbb{P}}=(\Re/\triangleleft)/\operatorname{Aut}^{\leftrightarrow}(\Re/\triangleleft,\triangleleft)$$
and let $\tilde{\preceq}$ be the relation induced by $\triangleleft$.
It is easy to see that $\tilde{\preceq}$ is a poset.
Set $\tilde{\Upsilon}=(\tilde{\mathbb{P}},\tilde{\preceq})$.
Note that by the definition we have that $\tilde{\mathbb{P}}$ is a super equivalence of $\Re/\triangleleft$ and, again, we think of $\tilde{\mathbb{P}}$ as an equivalence relation on $\GG$ that is coarser than $\Re/\triangleleft$.
Also note that if $P$ and $P'$ are contained in one $\tilde{\mathbb{P}}$ equivalence class, then either both $P,P'$ are from $\Re^{\leftrightarrow}$ or none.
Moreover, if $P\not=P'\in \Re^{\leftrightarrow}$ are $\tilde{\mathbb{P}}$ equivalent, then they are $\tilde{\preceq}$ incomparable.
In fact, this holds for every two $\Re/\triangleleft$ equivalence classes that are $\tilde{\mathbb{P}}$ equivalent because $\Re$ is finite.

\paragraph{Subpartition $(\Upsilon,\nabla)$}
Let $i\in \mathbb{Z}$.
We define $\nabla_i$ to be the subpartition induced by intersecting elements of $\Re^{\leftrightarrow}$ with the slice $G_\fI[i]$.
Similarly, we define $\mathbb{P}_i$ to be the super equivalence of $\nabla_i$ that is induced by $\tilde{\mathbb{P}}$ and $\preceq_i$ to be the poset structure induced by $\tilde{\preceq}$.
The important thing is that we are able to distinguish elements from $\tilde{\mathbb{P}}$, that is to say, each set in $\mathbb{P}_i$ can be index by the corresponding equivalence class from $\tilde{\mathbb{P}}$.
The situation for $\nabla_i$ is completely different, we are only able to say to what $\mathbb{P}_i$ class each $\nabla_i$ class belongs.

Let $\ta\in \Sigma_{out}\setminus \dom(\nabla_i)$.
By the definition of $\Re$ this means that every directed path that starts at $(i,\ta)$ eventually needs to visit some $P\in \Re^{\leftrightarrow}$.
By K\" onig's Theorem there is some $\ell_\ta\in \mathbb{N}$ such that this happen after at most $\ell_\ta$ steps (both to the left and right).
Set $\ell_i=\max_{\ta\not\in \dom(\nabla_i)}\ell_\ta$.

Let $(\sigma,\Upsilon,\nabla,\ell)$ be lexicographically minimal quadruple that appears in $G_\fI$ and let $C\subseteq \mathbb{Z}$ be those indices such that $(\fI(i),\Upsilon_i,\nabla_i,\ell_i)=(\sigma,\Upsilon,\nabla,\ell)$ for every $i\in C$.
By this we mean that $\nabla_i$ set-wise coincide with $\nabla$ and $\Upsilon_i$ index-wise correspond to $\Upsilon$.

In particular, $C\not=\emptyset$.
Note that if $C$ is not infinite in both directions, then we can pick a reference point in a Borel way, hence, solve the problem.
Assume that $C$ is infinite in both directions and let $C'\subseteq C$ be an $\ell$-separated subset.
Let $i<j\in C'$ be consecutive points and $B_{i,j}\subseteq \GG$ the block between them.

\begin{claim}
$B_{i,j}$ is a permutation $\sigma$-block on the subpartition $(\Upsilon,\nabla)$.
\end{claim}
\begin{proof}
Let $\ta,\tb\in \Sigma_{out}$ such that $(\ta,\tb)\in \type(B_{i,j})$.

{\bf (1.)}
Let $P\in \Re^{\leftrightarrow}$.
Define $\pi_{B_{i,j}}(P\cap \GG[i])=P\cap \GG[j]$.
It is easy to see that $\pi_{B_{i,j}}$ is $\mathbb{P}$-invariant because $\GG[i]\cap A=\GG[j]\cap A$ for every $A\in \widetilde{\mathbb{P}}$ by the definition of $\mathbb{P}$.

{\bf (2.)}
Fix some oriented path from $(i,\ta)$ to $(j,\tb)$.
It follows that $P\triangleleft Q$, where $(i,\ta)\in P$ and $(j,\tb)\in Q$.
This implies that $\mathbb{P}(\ta)\preceq \mathbb{P}(\tb)$ by the definition of $\Upsilon$.

{\bf (3.)}
Suppose that $\ta\not\in \dom(\nabla)$.
Fix some oriented path $p$ from $(i,\ta)$ to $(j,\tb)$.
By the definition we have that the length of $p$ is bigger than $\ell$, consequently, $p$ visits some $P\in \Re^{\leftrightarrow}$.
Since $(i,\ta)\not \in P'$ for every $P'\in \Re^{\leftrightarrow}$ we have that $\mathbb{P}(\ta)\not=\mathbb{P}(\tb)$.

{\bf (4.)}
Suppose that $\ta,\tb\in \dom(\nabla)$ and $\mathbb{P}(\ta)=\mathbb{P}(\tb)$.
Then we have $P\tilde{\preceq} Q$, where $(i,\ta)\in P$ and $(j,\tb)\in Q$.
However, since $P$ and $Q$ are $\tilde{\mathbb{P}}$ equivalent we must have $P=Q$ (non-equal elements are $\tilde{\preceq}$ incomparable).
\end{proof}

\paragraph{Picking an intertwined path}

Let $[\ta]_\nabla$ be the lexicographically minimal fix point of $\Gamma_{\sigma,(\Upsilon,\nabla)}$.
Let $P\in \Re^{\leftrightarrow}$ be the corresponding equivalence class in $\tilde{G}_{\fI}$.
Since $C'$ can be constructed in a Borel fashion and $[\ta]_{\nabla}$ is the same for every $i\in C'$ we have that $P$ can be chosen in a Borel fashion.

\paragraph{Solving the problem}
By previous paragraphs: we picked in a Borel way a set $C'\subseteq \mathbb{Z}$, $[\ta]_\nabla$ and equivalence class $P\in \Re^{\leftrightarrow}$, i.e., $P$ satisfies the conclusion of \cref{cl:InfSupInLine}.
%and $\tb\in [\ta]_\nabla$ whenever $(i,\tb)\in P$ and $i\in C'$.

Let $i\in C'$, then there is $(i-s_i,\tb_i)\in P$, where $s_i>0$, such that there is a directed path in $P$ from $(i-s_i,\tb_i)$ to every $(i,\tc)\in P$.
This is by iterative application of \cref{cl:InfSupInLine}.
\begin{figure}
    \centering
    \includegraphics[width=.95\textwidth]{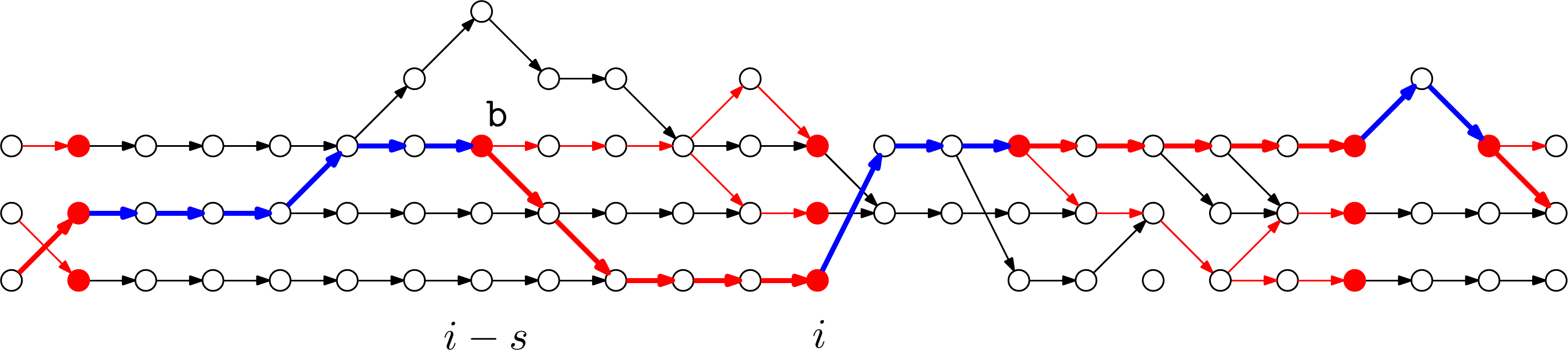}
    \caption{The figure shows how the final doubly infinite path is picked. For each $i \in C''$ we know that we can reach the letter $\tb$ at position $i-s$ (red arrows). We pick an arbitrary path from each $\tb$ at position $i-s$ to the left, until we reach another position from $C''$ (blue arrows). Then we pick the respective red path that connects to the previous $\tb$ letter. With this, every consecutive pair of $\tb$ letters is connected and the doubly infinite path constructed. }
    \label{fig:lines_gluing}
\end{figure}

Let $(s,\tb)$ be the lexicographically minimal pair that appears.
Again, if it does not appear infinitely many times in both directions we are done by picking a reference point.
Otherwise pass to a set $C''\subseteq C'$ that is at least $s$-separated.
This can be done in a Borel way.
Now pick any doubly infinite path that goes through $(i-s,\tb)$ for every $i\in C''$, see \cref{fig:lines_gluing}.
This finishes the proof.
\end{proof}

\subsection{Mixing implies global}

We show that if $\Pi$ is mixing, then it is not in $\measure$, $\baire$ nor $\fiid$.
This is achieved by translating $\Pi$ into a problem that is similar to a permutation problem, see \cref{def:PermProblem}.
Namely, consider all possible input labelings that consists of concatenating permutation $\sigma$-blocks on a subpartition $(\Upsilon,\nabla)$, where $\Pi$ is $(\sigma,\Upsilon,\nabla)$-mixing.
Let $\fI\in \Sigma_{in}^\mathbb{Z}$ be such a labeling and $f(\mathcal{I})\in \Sigma_{out}^\mathbb{Z}$ be a valid solution, i.e., $(\fI,f)$ is a $\Pi$-coloring.
Suppose moreover that $C\subseteq \mathbb{Z}$ is a set of indices that mark starting points of these blocks in $\fI$, this is not uniquely determined in general.
There are two possibilities (a) there is $\ta\in \Sigma_{out}$ such that $f(i)\in \mathbb{P}(\ta)$ for every $i\in C$, (b) $f$ is jumping between $\mathbb{P}$ classes.
Note that if (a) occurs, then $f(i)\in \dom(\nabla)$ for every $i\in C$ and once we know the value $[f(i')]_\nabla$ for one $i'\in C$, then all other values $[f(j)]_\nabla$ are fully determined for every $j\in C$ by (4) in \cref{def:subpartition}.
In the case (b) there are only finitely many indices in $C$ where $f$ can switch $\mathbb{P}$-class, by (2) in the \cref{def:subpartition}.
This allows to pick a reference point in a Borel way, i.e., the lexicographically minimal index where the switching occurs.
In another words, (b) can only happen on a ``small'' set of labelings.
On the other hand we show that property (a) is too strong for all the aforementioned models, on high-level, there need to be some correlation decay between decisions.

In the formal proof, we first define the space of mixing blocks, then endow this space with standard Borel structure, Borel probability measure and Polish topology, and finally show that there is no measurable solution in all these set-ups.

\paragraph{The space of mixing blocks}
Let $\Pi$ be an LCL and suppose that $\Pi$ is $(\sigma,\Upsilon,\nabla)$-mixing.
Pick finitely many $\sigma$-blocks $\{B_1,\dots,B_\ell\}$ on the subpartition $(\Upsilon,\nabla)$ such that the corresponding permutations $\pi_i:=\pi_{B_i}$ generate $\Gamma_{\sigma,(\Upsilon,\nabla)}$.

We define
$$\mathcal{K}\subseteq \Sigma_{in}^\mathbb{Z}\times (2^\mathbb{N})^\mathbb{Z}\times ([\ell]\times \{0,1\})^\mathbb{Z}$$
to be the space of oriented lines with nodes labeled by elements from $\Sigma_{in}$ and $2^{\mathbb{N}}$, together with the additional property that the $\Sigma_{in}$-labeling form a chain of blocks from $\{B_1,\dots,B_\ell\}$ and every node is marked with an element from $[\ell]\times \{0,1\}$ according to what block it belongs to, $[\ell]$, and whether it is a starting point of that block, $\{0,1\}$.
We make the convention that if $b(i)=(k,1)$, then $k$ refers to the block whose starting point $i$ is (in this case $i$ is also a last element of the previous block).
This convention induces a slight asymmetry in the statement of our main technical result \cref{thm:main_technical}.
Moreover, we assume that every element in $\mathcal{K}$ has no symmetry i.e., the canonical shift action is apperiodic.
To summarize:
If $(\fI,x,b)\in \mathcal{K}$ is given, then, for every $i\in \mathbb{Z}$, we have access to the $\Sigma_{in}$-label, $\fI(i)\in \Sigma_{in}$, the real label, $x(i)\in 2^{\mathbb{N}}$, and the index/position of/in the corresponding block $b(i)\in [\ell]\times \{0,1\}$.
In addition, $n\cdot (\fI,x,b)\not=(\fI,x,b)$ for every $n\in \mathbb{Z}\setminus \{0\}$ and it is easy to see that $\fK$ is invariant under the shift action.

The space $\mathcal{K}$ endowed with suitable measure or topology will serve as a witness to the fact that $\Pi$ is not in the class $\MEASURE$ nor $\BAIRE$.
From now on we switch to the language of shift actions, see \cref{subsec:shifts}.
Recall that $X_\Pi$ is the space of all $\Pi$-colorings.

\paragraph{Topology and Measure}

First, we endow $\mathcal{K}$ with a topological structure.
Since
$$\mathcal{K}\subseteq \Sigma_{in}^\mathbb{Z}\times (2^\mathbb{N})^\mathbb{Z}\times ([\ell]\times \{0,1\})^\mathbb{Z}$$
and the latter space carries a natural compact metrizable topology, we simply consider $\mathcal{K}$ as a topological subspace of it.
We write $(\fK,\tau)$ to stress the topological structure.
Recall that we assume that elements in $\mathcal{K}$ have no symmetries.

\begin{claim}
The space $(\mathcal{K},\tau)$ is a Polish topological space, that is, separable and completely metrizable.
The canonical shift action is an aperiodic Borel (continuous) automorphism and the map $i:\mathcal{K}\to \Sigma_{in}$, that is defined as $i(\fI,x,b)=\fI(0)$, is Borel (continuous).
\end{claim}

The $\sigma$-algebra of Borel sets generated by the topology $\tau$ turns $\fK$ into a standard Borel space.
We define a suitable Borel probability measure $\nu$ that turns $\fK$ into a standard probability space.

In fact, the measure $\nu$ is defined on $\Sigma_{in}^\mathbb{Z}\times (2^\mathbb{N})^\mathbb{Z}\times ([\ell]\times\{0,1\})^\mathbb{Z}$, then we show that it satisfies $\nu(\fK)=1$.
We consider the product measure of uniform measures, $\lambda$, on $(2^{\mathbb{N}})^\mathbb{Z}$.
Next we define a Borel probability measure $\mu$ on $\Sigma_{in}^\mathbb{Z}\times ([\ell]\times \{0,1\})^\mathbb{Z}$.
At the end we set $\nu=\lambda\times \mu$.

In order to construct the measure $\mu$, we define a Markov chain $\fM=(S,\fT)$.
Let $B_i$ be one of the permutation $\sigma$-blocks.
Formally $B_i=(\sigma,\sigma_2,\dots,\sigma_{\ell_i-1},\sigma)$ is a string of input labels, i.e., of elements from $\Sigma_{in}$, of length $|B_i|=\ell_i$.
Write $B'_i$ for a collection of $\alpha_{j,i}=(\sigma_j,j,i)$, where $j<\ell_i$, that is we consider $\Sigma_{in}$-labels that appear in $B_i$ together with the information about $B_i$ and their position, except for the last letter $\sigma$.
Set $S=\bigcup_{i\le \ell}B'_i$.

\begin{definition}
\label{def:chain}
Let $\mathcal{M}=(S,\fT)$ be a Markov chain with a state space $S$ and a transition relation $\fT$, see \cref{fig:chain}, that for every $i\in [\ell]$
\begin{enumerate}
    \item sends $\alpha_{j,i}$ to $\alpha_{j+1,i}$ with probability $1$ whenever $j<\ell_i-1$,
    \item sends $\alpha_{\ell_i-1,i}$ to $\alpha_{0,i'}$ with a uniform probability $\frac{1}{\ell}$ for every $i'\in [\ell]$.
\end{enumerate}
\end{definition}

\begin{figure}
    \centering
    \includegraphics[width=.6\textwidth]{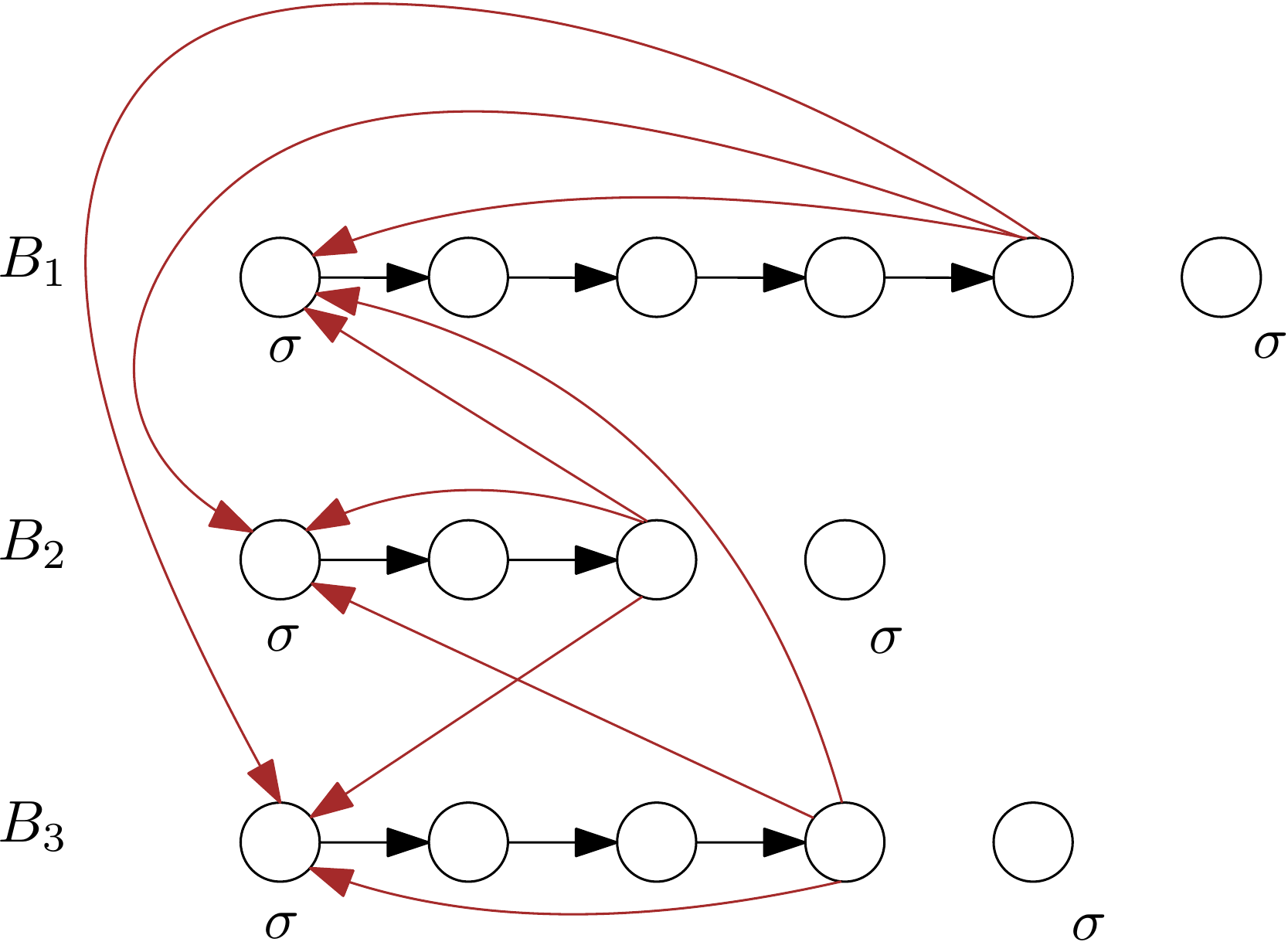}
    \caption{Illustration of the Markov chain from \cref{def:chain}. Brown arrows in the picture have transition probability $1/\ell = 1/3$, the black ones have transition probability $1$. }
    \label{fig:chain}
\end{figure}

Let $\mu'$ be a uniform distribution on $S$.
It is easy to see that $\mu'$ is stationary for $\fM$.
Let $\mu$ be the distribution on doubly infinite random walks given by $(\fM,\mu')$.

\begin{claim}
The Borel probability measure $\nu=\lambda\times \mu$ is concentrated on $\fK$.
The function $i:\fK\to \Sigma_{in}$ is measurable and the shift action $\mathbb{Z}\curvearrowright \fK$ is $\nu$-preserving.
\end{claim}

\paragraph{Mixing implies global (modulo technical result)}

Let
$$f:\mathcal{K}\to \Sigma^{\mathbb{Z}}_{out}$$
be a Borel equivariant map.
Define $\fL\subseteq \fK$ to be the set of those $(\fI,x,b)$ such that $f$ switches $\mathbb{P}$ equivalence classes on $(\fI,x,b)$.
Formally, $(\fI,x,b)\in \fL$ if there are $i,j\in \mathbb{Z}$ that are starting indices of some blocks. i.e., $b(i)=({-},1), \ b(j)=({-},1)$, and $\mathbb{P}(f(\fI,x,b)(i))\not=\mathbb{P}(f(\fI,x,b)(j))$.
Define $\fC\subseteq \fK$ to be the set of those $(\fI,x,b)$ such that $f(\fI,x,b)$ is $\Pi$-coloring.
Note that both $\fL$ and $\fC$ are shift invariant Borel sets.

\begin{claim}\label{cl:one_P_class}
Suppose that $f$ gives a $\Pi$-solution $\nu$-almost everywhere (or on a $\tau$-comeager set).
That is $\nu(\fC)=1$ (or $\fC$ is $\tau$-comeager).
Then $\nu(\fL)=0$ (or $\fL$ is meager).
\end{claim}
\begin{proof}
By the definition of $\fL$, $\fC$ and \cref{def:subpartition}, we define a Borel set $\fX\subseteq \fL\cap \fC$ that intersect every $\mathbb{Z}$-orbit of $\fL\cap \fC$ exactly in one point.
That is done as follows, let $(\fI,x,b)\in \fX$ if $(\fI,x,b)\in \fL$ and $b(0)=({-},1)$, i.e., $0$ is a left most point of a block, and $0$ is the left-most index when $\mathbb{P}$ equivalence class is changed.
Note that such an index exists since $f$ is a $\Pi$-solution on $\fL\cap \fC$ and by (2) in \cref{def:subpartition}.

By the assumption we have that $\fL\setminus \fC$ is $\nu$-null (or $\tau$-meager).
It remains to show that $\fX$ is such (since $\mathbb{Z}\cdot \fX=\fL\cap \fC$ and both notions are shift invariant).
It is a standard fact from descriptive combinatorics that since $\nu$ is $\mathbb{Z}$-invariant, we must have that $\fX$ is $\nu$-null (in the Baire case we use the fact that the shift action is $\tau$-continuous to deduce that $\fX$ is $\tau$-meager).
\end{proof}

From now on suppose that the assumption of $\cref{cl:one_P_class}$ is satisfied, i.e., $f$ is a $\Pi$-solution $\nu$-almost everywhere (or on a $\tau$-comeager set).
Write $\fS:=\fC\setminus \fL$, that is, $\fS$ is a Borel set on which $f$ solves $\Pi$ by always picking one $\mathbb{P}$ equivalence class and that satisfies $\nu(\fS)=1$ (or $\fS$ is $\tau$-comeager).

Define a Borel equivariant function 
$$\Lambda_f:\fS\to \nabla^\mathbb{Z}$$
as follows.
For $j\in \mathbb{Z}$, let $i\le j$ be first index such that $b(i)=(\_,1)$.
Set $\Lambda_f(\fI,x,b)(j)=[f(i)]_\nabla$.
Note that this is well-defined by \cref{def:subpartition} and our assumption on $\fS$.

We formulate separately the main technical result that is used in the proof of \cref{thm:main}.
The proof of this result is deferred to \cref{app:technical}.

\begin{restatable}{theorem}{maintechnical}\label{thm:main_technical}
Let $f:\fK\to \Sigma_{out}^\mathbb{Z}$ be a Borel equivariant map and $\fS\subseteq \fK$ be as above.
Then on $\nu$-conull subset of $\fS$ (or on a $\tau$-comeager subset of $\fS$) we have:
\begin{enumerate}
    \item the value $\Lambda_f(\fI,x,b)(0)$ depends only on $(\fI,x,b)\upharpoonright \{0,1,\dots\}$,
    \item the value $\Lambda_f(\fI,x,b)(0)$ depends only on $(\fI,x,b)\upharpoonright \{\dots -2,-1\}$.
\end{enumerate}
E.g. in the measure case, there is a $\nu$-conull set $C\subseteq \fS$ such that if $(\fI,x,b),(\fI',x',b')\in C$ agree on $\{0,1,\dots\}$, then $\Lambda_f(\fI,x,b)(0)=\Lambda_f(\fI',x',b')(0)$.
\end{restatable}

We remark that this statement does not use the assumption that $\Pi$ is mixing, i.e., this holds for any LCL $\Pi$ and any permutation $\sigma$-blocks $\{B_1,\dots,B_\ell\}$.

Recall, however, that we assume that $\Pi$ is mixing.
We show how \cref{thm:main_technical} implies that there is no Borel equivariant solution of $\Pi$ on $\nu$-conull (or $\tau$-comeager) set.
This immediately yields that $\Pi\not\in \baire,\measure$.
Moreover, once we have this consider a map $\gamma:\fK\to \Sigma_{in}^\mathbb{Z}\times (2^{\mathbb{N}})^\mathbb{Z}$ that forgets the third coordinate.
This map is clearly Borel.
It follows that if $\Pi$ were in the class $\FIID$, then composing any Borel equivariant map that witnesses that $\Pi\in \fiid$ with $\gamma$ would yield a Borel equivariant map on $\fK$ that solves $\Pi$ on $\nu$-conull set by Fubini's Theorem.
Consequently, $\Pi\not\in \FIID$.

%Note that we assume that $\Pi$ is mixing.
%This shows that $\Pi\not\in \baire,\measure$.

We only show the measure case, as the Baire case is analogous.
Assume that there is an equivariant Borel map $f$ and $\fS$ as above.
Observe that the assumption that $\Pi$ is mixing together with the fact that $\nu$ is $\mathbb{Z}$-invariant imply that there are $\ta,\tb\in \Sigma_{out}$ such that $[\ta]_\nabla\not=[\tb]_\nabla$ and $\nu(\Lambda_f^{-1}([\ta]_\nabla)),\nu(\Lambda_f^{-1}([\tb]_\nabla))>0$.
Let $A_\ta:=\Lambda_f^{-1}([\ta]_\nabla)$ and $C\subseteq \fS$ be the $\nu$-conull set from \cref{thm:main_technical}.
It follows that $\nu(A_\ta\cap C)>0$.
Since $\Lambda_f$ is invariant under modification of values of $\{\dots,-2,-1\}$ and $\{0,1,\dots\}$ in $C$, we must have $\nu(A_\ta)=1$ by Fubini's Theorem.
That contradicts the fact that $\nu(\Lambda_f^{-1}([\tb]_\nabla))>0$.

\begin{remark}
To find a $\Pi$-coloring for a given input labeling $\fI\in \Sigma_{in}^\mathbb{Z}$ is, up to a small technical nuance, equivalent with picking in a Borel way one of the $\Re^{\leftrightarrow}$ equivalence classes (these are defined in the proof of \cref{cl:paths_upper_bound}).
Abstractly, this can be phrased as follows: let $(X,\fB)$ be a standard Borel space and $E,F$ be (countable) Borel equivalence relations such that $F$ has finite (uniformly bounded) index in $E$. Under what circumstances there is a Borel set $B\subseteq X$ that intersect each $E$ equivalence class in exactly one $F$ equivalence class?
This was problem was studied in different, more abstract, perspective by de Rancourt and Miller \cite{MillerNoeDichotomy}.
\end{remark}

\subsection{Decidability}

It remains to check that the combinatorial property of $\Pi$ being $(\sigma,\Upsilon,\nabla)$-mixing is decidable. 
This is done similarly to the decidability result from \cite{balliu2019LCLs_on_paths_decidable}. 
In what follows, we assume that the problem $\Pi$ is given on the input in its normal form. Instead, it may be given in its general form (\cref{def:lcls_on_paths}) by listing all possible $r$-hop neighbourhoods and stating whether it is in $\fP$ for each one of them. The reduction to the normal form then runs in time polynomial in the length of the input. 

\begin{claim}
\label{cl:paths_decidability}
Verifying whether $\Pi$ is $(\sigma, \Upsilon, \nabla)$-mixing for some $\sigma$ and a subpartition $(\Upsilon, \nabla)$ is in $\pspace$. 
\end{claim}
\begin{proof}

Given the description of $\Pi$, it can be decided whether it is $(\sigma, \Upsilon, \nabla)$-mixing for some $\sigma$ and $(\Upsilon, \nabla)$ as follows. 
We iterate over all $\sigma \in \Sigma_{in}$ and subpartitions $(\Upsilon, \nabla)$. We need to generate all types of $\sigma$-blocks on a subpartition $(\Upsilon, \nabla)$ and check whether the corresponding group $\Gamma_{\sigma, (\Upsilon, \nabla)}$ has a fixed point. 
%As a warm-up, we first show how to check it with an algorithm that works in exponential space. 

Importantly, note that the set of all $\sigma$-blocks of a given type is potentially infinite but we claim that the set of all $\sigma$-blocks on a subpartition $(\Upsilon, \nabla)$ of a given type contains a block of length at most $|\Sigma_{in}| \cdot  2^{|\Sigma_{in}| \cdot |\Sigma_{out}|}$. 
To see it, consider any permutation $\sigma$-block $B = \sigma=\sigma_1\sigma_2\dots\sigma_\ell=\sigma$ and any of its prefixes $B_i = \sigma_1\dots\sigma_i$. 
We will say that two prefixes $B_i$ and $B_j$ are of the same \emph{reachability type} if $\sigma_i = \sigma_j$ and for every $\tau_1, \tau_2 \in \Sigma_{out}$ we have $(0,\tau_1) \sim (i,\tau_2)$ if and only if $(0,\tau_1) \sim (j,\tau_2)$, where $\sim$ denotes the fact that there is directed path connecting these nodes. 
If there are two prefixes $B_i$ and $B_j$ of the same reachability type for $i < j$, we can turn $B$ into $B' = \sigma_1\sigma_2\dots\sigma_i\sigma_{j+1}\dots\sigma_\ell$. This new $\sigma$-block $B'$ is of the same type as the original block $B$.  
Now note that while $|B| > |\Sigma_{in}| \cdot  2^{|\Sigma_{in}| \cdot |\Sigma_{out}|}$, we can find two different prefixes $B_i,B_j$ of the same reachability type by pigeonhole principle. Hence, for any block $B$, there exists a block of the same type with length at most $|\Sigma_{in}| \cdot  2^{|\Sigma_{in}| \cdot |\Sigma_{out}|}$. 

Hence, the $(\sigma,\Upsilon, \nabla)$-mixing property for particular $\sigma$ and $(\Upsilon, \nabla)$ can be checked by generating all permutation blocks of at most exponential length. This means that the existence of $\sigma$ and $(\Upsilon, \nabla)$ such that $\Pi$ has the $(\sigma,\Upsilon, \nabla)$-mixing property can also be decided in space exponential in $|\Pi| := |\Sigma_{out}| + |\Sigma_{in}|$ (this quantity is in polynomial relation to the length of the description of the problem when given in the normal form). We will now show how to improve this to space \emph{polynomial} in $|\Pi|$. 

To see that, note that by Savitch's theorem\cite{savitch1970relationships,arora_barak2009computational_complexity_modern_approach} it suffices to show that the problem is in $\mathsf{NPSPACE}$. 
That is, it suffices to check a certificate that $\Pi$ is $(\sigma, \Upsilon, \nabla)$-mixing. The certificate can be of exponential size, though it can be accessed from a tape that we can read only once in a left to right manner. 
The certificate we use to prove that $\Pi$ is $(\sigma,\Upsilon, \nabla)$-mixing for some $\sigma$ and $(\Upsilon, \nabla)$ is first $\sigma$ and $(\Upsilon, \nabla)$ and then a sequence of at most $|\Sigma_{out}|$ permutation $\sigma$-blocks on the subpartition $(\Upsilon, \nabla)$, each one proving that a different $\tau \in \nabla$ is not a fixed point of $\Gamma_{\sigma, (\Upsilon, \nabla)}$. 
Each block is represented by an exponentially sized string. Although this string does not fit into the polynomial memory, we can scan it from the nondeterministic tape and during the scanning for each pair of letters $\ta$ from the beginning slice of the block and $\tb$ from the last slice of the block, we can simultaneously compute whether $\ta$ and $\tb$ are connected. This is done by a simple dynamic program that for the current slice computes this connectivity property and extends this relation to the new slice whenever it arrives.
Hence, we can check that the block given in the input is a valid permutation $\sigma$-block on subpartition $(\Upsilon, \nabla)$ for which $\tau$ is not a fixed point. 
\end{proof}
It is not clear whether checking $\Pi \in \local(O(1))$ or $\Pi \in \local(O(\log^* n))$ is in $\pspace$ \cite{balliu2019LCLs_on_paths_decidable}. 

\begin{claim}
\label{cl:paths_hardness}
Verifying whether $\Pi$ is $(\sigma, \Upsilon,\nabla)$-mixing for some $\sigma$ and $(\Upsilon,\nabla)$ is $\pspace$-hard, if $\Pi$ is given in the input in its normal form.  
\end{claim}

We provide only a proof sketch here, as the hardness result of \cite{balliu2019LCLs_on_paths_decidable} already almost applies to our setup. We only discuss what needs to be changed in their proof. 

\begin{proof}[Proof Sketch]
In \cite{balliu2019LCLs_on_paths_decidable} it is proven that distinguishing between the case $\Pi \in \local(O(1))$ and $\Pi \in \local(\Omega(n))$ is $\pspace$-hard. The problem $\Pi$ used there encodes a simulation of a linear bounded automaton as an LCL problem on an oriented path. There is a block of $O(B)$ nodes in the input of the LCL, each having some input, that simulates one step of the automaton on some input letter. If the automaton always accepts a string, the problem is in $\local(O(1))$, because the automaton has to accept after number of steps exponential in $B$, by the argument from \cref{cl:paths_decidability}. Hence, the local complexity is $O(1)$. If the automaton loops given some string, the LCL problem has $\Omega(n)$ complexity since we add the following feature to the encoding: Before the first step of the automaton, we give as input a ``secret'' bit. Then we enforce with local constraints that the same bit needs to be outputted at the beginning of each  given at the beginning of the execution needs to be a part of output of all nodes that encode the simulation problem. This means that if the automaton loops, there is an arbitrarily long input such that the solution of the LCL problem labels all nodes encoding the input with the secret bit. This problem has local complexity $\Omega(n)$. 

We now discuss how to adapt the above proof to our setup. Instead of making the problem hard by adding the ``secret'' bit, we make it hard by incorporating the permutation problem from \cref{def:PermProblem}. After each block encoding one step of the automaton, we add either the identity or the swap permutation as the input. All nodes of a block have to additionally output either $0$ or $1$, such that all nodes of a given block output the same bit and if neighboring blocks are joined with the identity permutation, they have to output the same bit, and if they are joined with the swap permutation, they need to output different bits. 
If the original automaton $\fA$ always finishes, this problem is still solvable in $O(1)$ local rounds. However, if $\fA$ loops on some input $I$, we can construct a hard instance as in \cref{fig:hard_input}. That is, the blocks are grouped in \emph{superblocks} of length $|I|$ that represent the input $I$. Moreover, neighboring blocks are joined by the identity permutation if they are from the same superblock and with the swap permutation otherwise. 
Any valid solution to the problem solves $2$-coloring on the path that we get by contracting superblocks. But this problem does not have a Borel solution and this finishes the proof. 

\begin{figure}
    \centering
    \includegraphics[width=.9\textwidth]{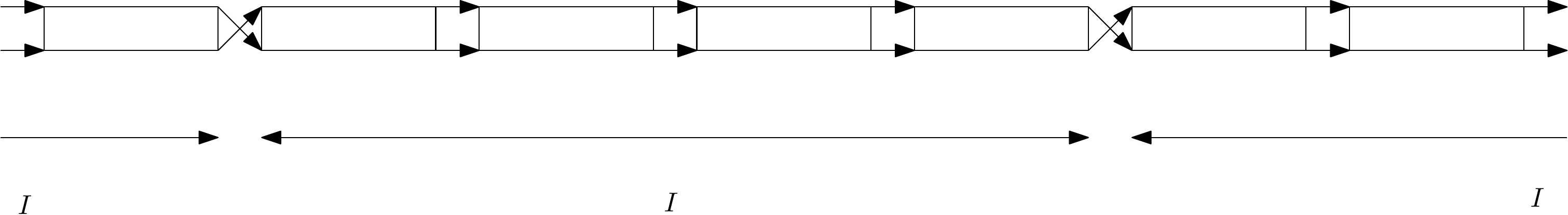}
    \caption{The input is a sequence of superblocks, each corresponding to the string $I$ that makes $\fA$ loop. }
    \label{fig:hard_input}
\end{figure}
\end{proof}

\section{Open Problems}\label{sec:Open_Problems_lines}

Here we list some open problems that we find interesting. 

\begin{enumerate}
    \item In case of paths with inputs, we have that all problems in the class $\borel$ can in fact be constructed by iterating MIS constructions an $\local(O(1))$ algorithms. Is true in general on bounded degree graphs? 
    
    \item Is $\borel = \measure = \fiid$ on any graph class of subexponential growth? 

    \item Is the classification of \cref{thm:main} $\pspace$-complete?
\end{enumerate}

\section*{Acknowledgments}
We thank to Anton Bernshteyn, Sebastian Brandt, Yi-Jun Chang, Mohsen Ghaffari, Stephen Jackson, Brandon Sewart, Jukka Suomela, Zoltán Vidnyászky for many engaging discussions. In particular, we thank Jukka for suggesting the proof of \cref{cl:paths_hardness}. 
The first author was supported by Leverhulme Research Project Grant RPG-2018-424. 
This project has received funding from the European Research Council (ERC) under the European
Unions Horizon 2020 research and innovation programme (grant agreement No. 853109). 

\bibliographystyle{alpha}
\bibliography{ref}

\appendix

\section{Proof of \cref{thm:main_technical}}\label{app:technical}

We recall the statement for the convenience of the reader.

\maintechnical*

The high-level strategy is to exploit the fact that every measurable function can be finitely approximated.
In the Baire case this is achieved by the fact that every such function is continuous on a comegear set, while in the measure case we use Doob's martingale convergence theorem.
We prove both cases separately and we start with the easier case, Baire.
Also we note that it is enough to show (1.) since the argument for (2.) is the same up to the small difference given by the definition of $b$ in $(\fI,x,b)\in \fK$, namely, $b$ assigns right-most points in a block to the next block on the right.
First we introduce some notation that we use in both cases.

Suppose that $f$, $\fS$ and $\Lambda_f$ are given.
Formally, we want $\Lambda_f$ to be defined everywhere, i.e., extend $\Lambda_f$ in a Borel equivariant fashion anyhow to $\fK\setminus \fS$.
We set $\Lambda:=\Lambda_f$ and let $\lambda$ be $\Lambda$ evaluated at $0$.
We denote the equivalence classes of $\nabla$ by $\tA,\tB,\dots$.
Let $\fB$ be the Borel $\sigma$-algebra on $\fK$.
Let $t\in \mathbb{N}$.
Write $\fB_t$ for the (finite) $\sigma$-algebra generated by $t$-neighborhoods, where a $t$-neighborhood of $(x,\fI,b)\in \fK$ is the collection of all $(x',\fI',b')\in \fK$ such that $x(v)=x'(v)$ on first $t$-many bits, $\fI(v)=\fI'(v)$ and $b(v)=b'(v)$ for every  $v\in [-t,\dots,t)$.
It is easy to see that $\fB_t\subseteq \fB_{t+1}$ and that the minimal $\sigma$-algebra that contains $\bigcup_{t\in \mathbb{N}} \fB_t$ is the Borel $\sigma$-algebra $\fB$ on $\fK$.
Similarly we define $\fB_t^+$ to be the (finite) $\sigma$-algebra generated by right $2t$-neighborhoods, that is $\fB_t$ moved to the right by $t$, and by $\fB^+$ the $\sigma$-algebra generated by $\bigcup_{t\in \mathbb{N}} \fB^+_t$.

\paragraph{$\baire$}

Since $\fK$ is a Polish space, we find a comeager set $C\subseteq \fK$ such that $\lambda$ is continuous when restricted to $C$, see \cite[Theorem~18.6]{KecClassic}.
Let $\tA\in \nabla$ and $\fX$ be a $t$-neighborhood, i.e., an atom in $\fB_t$, such that
\begin{equation}
    \fX\cap C\subseteq \lambda^{-1}(\tA).
\end{equation} 
These exists by continuity of $\lambda$ together with the definition of the topology $\tau$ on $\fK$.
Consider now the shifts of $\fX$ to the right, that is, for every $i>0$ let $\fX_i$ be the shift of $\fX$ by $-(i+t)$.
Then we have $\fX_i\in \fB^+$ for every $i>0$.
Observe that whenever $(x,\fI,b)\in \fX_i\cap C\cap \fS$, for some $i>0$, then $\lambda(x,\fI,b)$ is determined by restriction of $(x,\fI,b)$ on $\{0,1,\dots\}$.
This follows from the assumption that $f$, and therefore $\Lambda$, is equivariant together with the definition of $\fS$.
The following claim finishes the proof, the desired comeager set that satisfies (1.) is $C\cap D\cap \fS$.

\begin{claim}
The set
$$D=\{(x,\fI,b)\in \fK:\exists i>0 \ (x,\fI,b)\in \fX_i\}$$
is comeager in $\fK$.
\end{claim}
\begin{proof}
In fact, we show that $D$ is open dense, that is, $D$ is open and intersect every non-empty open subset of $\fK$.
Since every dense open set is comeager, the claim follows.

It follows from continuity of the shift action together with the definition of the topology $\tau$, i.e., $\fX$ is open, that $D$ is open.
Let $\fY$ be any $t'$-neighborhood.
Taking $i>0$ large enough allows to merge $\fY$ and $\fX_i$ into some $t''$-neighborhood $\fZ$.
In particular, since $\fZ$ is nonemtpy, we have $\fY\cap D\not=\emptyset$.
Consequently, $D$ intersects every open set and we are done.
\end{proof}

\paragraph{$\measure$}
The situation for measure is more complicated, this is because there is no single big ($\nu$-conull) set on which $\lambda$ is continuous.

It will be more convenient to change the perspective and to view $\lambda$ as a function that assigns to each $(x,\fI,b)\in \fK$ a distribution on $\nabla$.
That is if $\lambda(x,\fI,b)=\tA\in\nabla$, then we may view it as the Dirac measure concentrated on the atom $\{\tA\}$.
Define
$$\lambda_t=\E(\lambda|\fB_t)$$
to be the conditional expectation of $\lambda$ given $\fB_t$.
In another words, if $\fX$ is a $t$-neighborhood of $(x,\fI,b)$, then $\lambda_t((x,\fI,b))$ is the average of the values of $\lambda$ on $\fX$.
Similarly we define $\lambda^+_t=\E(\lambda|\fB^+_t)$ and $\lambda^+=\E(\lambda|\fB^+)$.
Our goal is to show that almost surely we have $\lambda^+=\lambda$.

By Doob's Martingale Convergence Theorem we have
\begin{equation}\label{eq:Doob}
\begin{split}
    \lambda(x,\fI,b)= & \ \lim_{t\to\infty} \lambda_t(x,\fI,b) \\
    \lambda^+(x,\fI,b)= & \ \lim_{t\to\infty} \lambda^+_t(x,\fI,b)
\end{split}
\end{equation}
$\nu$-almost surely.

\begin{claim}
The value of $\lambda^+$ is a Dirac measure $\nu$-almost surely.
\end{claim}
\begin{proof}
First line in \eqref{eq:Doob} together with the definition of $\lambda$ implies that for every $\epsilon>0$ there is $t>0$ and a measurable set $A_\epsilon\in \fB_t$ such that $\nu(A_\epsilon)>1-\epsilon$ and $\lambda_t(x,\fI,b)$ is $\epsilon$-close to a Dirac measure for every $(x,\fI,b)\in A_\epsilon$.

By the equivariance of $\Lambda$, the definition of $\fS$ and the definition of $\fB_t^+$ we have that the same holds for $\lambda^+_t$.
That is $\lambda_t^+(x,\fI,b)$ is $\epsilon$ close to a Dirac measure for every $(x,\fI,b)\in ((-t)\cdot A_\epsilon)$, where $\nu((-t)\cdot A_\epsilon)=\nu(A_\epsilon)>1-\epsilon$ since the shift action is $\nu$-preserving and there is a full dependence between the values of $\Lambda(\fI,x,b)(t)$ and $\Lambda(\fI,x,b)(0)$ on a $\nu$-conull set $\fS$.
This implies the claim because by the second line in \eqref{eq:Doob} we have that $\lambda^+_t\to \lambda^+$ a.s.
\end{proof}

Since we have $\lambda^+=\E(\lambda|\fB^+)$ and $\lambda$ always outputs a Dirac measure, it must be the case that $\lambda=\lambda^+$ holds $\nu$-almost surely.
This finishes the proof of \cref{thm:main_technical}.

\end{document}

%% file: macros.tex
\usepackage{amscd}
\usepackage{mathrsfs}
\usepackage[IL2]{fontenc}

\usepackage{comment} % \begin{comment}\end{comment} to comment stuff that may be uncommented later on

\usepackage{color}
\usepackage{amssymb,amsmath,amsthm,amsfonts,latexsym}
\usepackage[utf8x]{inputenc}
\usepackage{bbm} %this is so we can use mathbb with numbers, but using the command bbm

\usepackage[margin=1.5in]{geometry}
\usepackage{graphicx}
\usepackage[disable]{todonotes}
\usepackage{amsmath,amsfonts,amssymb,amsthm} 
\usepackage{tikz}
\usepackage{hyperref}
\usepackage{thmtools}
\usepackage{thm-restate}
\newtheorem{theorem}{Theorem}[section]

\newtheorem*{claim*}{Claim}
\newtheorem*{theorem*}{Theorem}
\newtheorem*{definition*}{Definition}

\newtheorem{corollary}[theorem]{Corollary}

\newtheorem{remark}[theorem]{Remark}
\newtheorem{claim}[theorem]{Claim}
\newtheorem{fact}[theorem]{Fact}

\newtheorem{observation}[theorem]{Observation}

\newtheorem{definition}[theorem]{Definition}

\newcommand{\comm}[1]{}
\usepackage[capitalize, nameinlink]{cleveref}
\crefname{theorem}{Theorem}{Theorems}
\crefname{proposition}{Proposition}{Propositions}
\crefname{observation}{Observation}{Observations}
\crefname{lemma}{Lemma}{Lemmas}
\crefname{claim}{Claim}{Claims}
\crefname{problem}{Problem}{Problems}
\crefname{conjecture}{Conjecture}{Conjectures}
\crefname{question}{Question}{Questions}
\crefname{example}{Example}{Examples}
\crefname{fact}{Fact}{Facts}

\newcommand{\dom}{\operatorname{dom}}

\newcommand{\local}{\mathsf{LOCAL}}

\newcommand{\rlocal}{\mathsf{R-LOCAL}}

\newcommand{\CONT}{\mathsf{CONTINOUS}}
\newcommand{\unifCONT}{\textrm{uniformly}\;\mathsf{CONTINOUS}}

\newcommand{\BOREL}{\mathsf{BOREL}}
\newcommand{\BAIRE}{\mathsf{BAIRE}}
\newcommand{\MEASURE}{\mathsf{MEASURE}}
\newcommand{\FIID}{\mathsf{fiid}}
\newcommand{\fiid}{\mathsf{fiid}}
\newcommand{\borel}{\mathsf{BOREL}}

\newcommand{\baire}{\mathsf{BAIRE}}
\newcommand{\measure}{\mathsf{MEASURE}}
\newcommand{\type}{\texttt{type}}

\newcommand{\tail}{\mathsf{TAIL}}

\renewcommand{\P}{\textrm{P}}
\newcommand{\E}{\textrm{\textbf{E}}}

\newcommand{\poly}{\operatorname{poly}}

\newcommand{\pspace}{\mathsf{PSPACE}}

\newcommand{\fA}{\mathcal{A}}
\newcommand{\fB}{\mathcal{B}}
\newcommand{\fC}{\mathcal{C}}

\newcommand{\fE}{\mathcal{E}}
\newcommand{\Epsilon}{\fE}

\newcommand{\fI}{\mathcal{I}}

\newcommand{\fK}{\mathcal{K}}
\newcommand{\fL}{\mathcal{L}}
\newcommand{\fM}{\mathcal{M}}

\newcommand{\fP}{\mathcal{P}}

\newcommand{\fS}{\mathcal{S}}
\newcommand{\fT}{\mathcal{T}}

\newcommand{\fX}{\mathcal{X}}
\newcommand{\fY}{\mathcal{Y}}
\newcommand{\fZ}{\mathcal{Z}}
\newcommand{\tA}{\mathtt{A}}
\newcommand{\tB}{\mathtt{B}}
\newcommand{\tC}{\mathtt{C}}

\newcommand{\tG}{\mathtt{G}}

\newcommand{\tR}{\mathtt{R}}

\newcommand{\ta}{\mathtt{a}}
\newcommand{\tb}{\mathtt{b}}
\newcommand{\tc}{\mathtt{c}}
\newcommand{\td}{\mathtt{d}}
\newcommand{\te}{\mathtt{e}}

\newcommand{\markkth}{$\textsf{mark}$-$k\textsf{th}$-$\textsf{element}$\;}
\newcommand{\blockcol}{$\mathsf{3\text{-}coloring\text{-}of\text{-}blocks}$\;}

\def\leukfrac#1/#2{\leavevmode
               \kern.1em
                \raise.9ex\hbox{\the\scriptfont0 ${}_#1$}
                \hskip -1pt\kern-.1em
                /\kern-.15em\lower.10ex\hbox{\the\scriptfont0 ${}_#2$}}

\makeatletter
\def\diam{\mathop{\operator@font diam}\nolimits}
\makeatother